\documentclass[12pt]{article}

\usepackage[a4paper,margin=2cm]{geometry}
\usepackage{latexsym}
\usepackage{amssymb}
\usepackage{amsmath}
\usepackage{amsthm}
\usepackage{booktabs}
\usepackage{enumitem}
\usepackage{graphicx}
\usepackage{color}
\usepackage{algorithm,algpseudocode}
\usepackage{array}
\usepackage{arydshln}
\usepackage[hidelinks]{hyperref}
\usepackage{cleveref}
\usepackage{multirow}
\usepackage{thm-restate}
\usepackage{url}
\usepackage{tikz}
\usepackage{natbib}
\usepackage[utf8]{inputenc}
\usepackage[english]{babel}
\usepackage[small]{caption}
\usepackage{subcaption}
\usepackage[T1]{fontenc}
\usetikzlibrary{arrows.meta,positioning,decorations.pathreplacing}

\newtheorem{theorem}{Theorem}[section]
\newtheorem{lemma}[theorem]{Lemma}
\newtheorem{corollary}[theorem]{Corollary}
\newtheorem{proposition}[theorem]{Proposition}
\theoremstyle{definition}

\newtheorem{example}{Example}

\renewcommand{\emptyset}{\varnothing}
\newcommand{\centered}[1]{\begin{tabular}{c} #1 \end{tabular}}
\newcommand{\cev}[1]{\reflectbox{\ensuremath{\vec{\reflectbox{\ensuremath{#1}}}}}}
\allowdisplaybreaks

\title{
On Connected Strongly-Proportional Cake-Cutting%
\footnote{
    A preliminary version of this work appeared in Proceedings of the 27th European Conference on Artificial Intelligence \citep{janko2024connectedecai}.
    This version contains new results on stronger than strongly-proportional allocations (\Cref{sec:stronger}) and pies (\Cref{sec:pies}), a discussion of why the characterization for hungry agents cannot be generalized to non-hungry agents (\Cref{sec:general}), as well as all proofs omitted from the conference version (\Cref{thm:general_condition,thm:general_improved_ub,thm:general_equal_lb}).
}
}

\author{
Zsuzsanna Jank\'{o}\footnote{Corvinus University of Budapest, HUN-REN Centre for Economic and Regional Studies}
\and
Attila Jo\'{o}\footnote{University of Hamburg}
\and
Erel Segal-Halevi\footnote{Ariel University}
\and
Sheung Man Yuen\footnote{National University of Singapore}
}

\begin{document}

\maketitle

\begin{abstract}
We investigate the problem of fairly dividing a divisible heterogeneous resource, also known as a cake, among a set of agents who may have different entitlements.
We characterize the existence of a connected strongly-proportional allocation---one in which every agent receives a contiguous piece worth strictly more than their proportional share.
The characterization is supplemented with an algorithm that determines its existence using $O(n \cdot 2^n)$ queries.
We devise a simpler characterization for agents with strictly positive valuations and with equal entitlements, and present an algorithm to determine the existence of such an allocation using $O(n^2)$ queries.
We provide matching lower bounds in the number of queries for both algorithms.
When a connected strongly-proportional allocation exists, we show that it can also be computed using a similar number of queries.
We also consider the problem of deciding the existence of a connected allocation of a cake in which each agent receives a piece worth a small fixed value more than their proportional share, and the problem of deciding the existence of a connected strongly-proportional allocation of a pie.
\end{abstract}

\section{Introduction}
\label{sec:intro}

Consider a group of siblings who inherited a land estate, and would like to divide it fairly among themselves.
The simplest procedure for attaining a fair division is to sell the land and divide the proceeds equally; this procedure guarantees each sibling a proportional share of the total land value.

But in some cases, it is possible to give each sibling a much better deal. 
As an example, suppose that the land estate contains one part that is fertile and arable, and one part that is barren but has potential for coal mining.
This land is to be divided between two siblings, one of whom is a farmer and the other is a coal factory owner.
If we give the former piece of land to the farmer and the latter piece of land to the coal factory owner, both siblings will feel that they receive more than half of the total land value. 
Our main question of interest is: when is such a superior allocation possible?

We study this question in the framework of \emph{cake-cutting}.
In this setting, there is a divisible resource called a \emph{cake}, which can be cut into arbitrarily small pieces without losing its value.
The cake is represented simply by an interval which can model a one-dimensional object, such as time.
There are $n$ agents, each of whom has a personal measure of value over the cake.
The goal is to partition the cake into $n$ pieces and allocate one piece per agent such that the agents feel that they receive a ``fair share'' according to some fairness notion.

A common fairness criterion---nowadays called \emph{proportionality}---requires that each agent $i$ receives a piece of cake that is worth, according to $i$'s valuation, at least $1/n$ of the total cake value.
In his seminal paper, \citet{steinhaus1948problem} described an algorithm, developed by his students Banach and Knaster, that finds a proportional allocation; moreover, this allocation is \emph{connected}---each agent receives a single contiguous part of the cake. 
This algorithm is now called the \emph{last diminisher} algorithm.

But the guarantee of proportionality allows for the possibility that each agent receives a piece worth \emph{exactly} $1/n$; when this is the case, there is little advantage in using a cake-cutting procedure over selling the land and giving $1/n$ to each partner.
A stronger criterion, called \emph{strong-proportionality} or \emph{super-proportionality}, requires that each agent $i$ receives a piece of cake worth \emph{strictly more} than $1/n$ of the total cake value from $i$'s perspective.
This raises the question of when such a strongly-proportional allocation exists.

Obviously, a strongly-proportional allocation does not exist when all the agents' valuations are identical, since if any agent receives more than $1/n$ of the cake, then some other agent must receive less than $1/n$ of the cake.
Interestingly, in all other cases, a strongly-proportional allocation exists.
Even when \emph{two} agents have non-identical valuations, there exists an allocation in which \emph{all} $n$ agents receive more than $1/n$ of the total cake value from their perspectives \citep{dubins1961cut,rebman1979get}.
\citet{woodall1986note} presented an algorithm for finding such a strongly-proportional allocation.
\citet{barbanel1996game} generalized this algorithm to agents with unequal entitlements, and \citet{janko2021cutting} presented a simple algorithm for this generalized problem and extended it to infinitely many agents.

The problem with all these algorithms is that, in contrast to the last diminisher algorithm for proportional cake-cutting, they do not guarantee a \emph{connected} allocation.
Connectivity is an important practical consideration when allocating cakes; for example, if the cake is the availability of a meeting room by time and needs to be allocated to different teams throughout the day, then a two-hour slot is easier for a team to utilize than six disjoint twenty-minute slots.
Indeed, connectivity is the most commonly studied constraint in cake-cutting literature \citep{suksompong2021constraints,stromquist1980how,su1999rental,stromquist2008envy,goldberg2020contiguous,elkind2022mind}, and relaxing this constraint may present each agent instead with a ``countable union of crumbs'' \citep{stromquist1980how}.

Thus, our main questions of interest are:

\begin{quote}
\emph{What are the necessary and sufficient conditions for the existence of a connected strongly-proportional cake allocation? What are the query complexities to determine these conditions?}
\end{quote}

\subsection{Our Results}
\label{sec:results}

The cake to be allocated, modeled by a unit interval $[0, 1]$, is to be divided among $n$ agents who may have different entitlements for the cake, with the entitlements summing to $1$.
Each agent receives an interval of the cake that is disjoint from the other agents' intervals.
Each agent has a valuation function on the intervals of the cake that is non-negative, finitely additive, and continuous with respect to length.
In this regard, the value of a single point is zero to every agent, and we can assume without loss of generality that agents receive \emph{closed} intervals of the cake, and that any two agents' pieces can possibly intersect at the endpoints of their respective intervals.%
\footnote{This is often assumed in cake-cutting literature; see e.g. \citet{procaccia2016handbook}.}
In order to access agents' valuations in the algorithms, we allow algorithms to make eval and (right-)mark%
\footnote{
  We choose \emph{right}-mark instead of the usual \emph{left}-mark for convenience.
  Our algorithms still work if only left-mark queries are available (together with eval).
  See \Cref{ap:left_right_marks} for a more detailed explanation.
}
queries of each agent as in the standard Robertson-Webb model \citep{robertson1998cake}.
More details of our model are provided in \Cref{sec:model}.

In \Cref{sec:hungry}, we consider \emph{hungry} agents---those who have positive valuations for any part of the cake with positive length.
For agents with equal entitlements, we show that a connected strongly-proportional allocation exists if and only if there are two agents with different $r$-marks for some $r \in \{1/n, 2/n, \ldots, (n-1)/n\}$, where an $r$-mark is a point that divides the cake into two such that the left part of the cake is worth $r$ to that agent.
This implies that the existence of such an allocation can be decided using $n(n-1)$ queries.
The proof of sufficiency is constructive, so a connected strongly-proportional allocation can be computed using $O(n^2)$ queries if it exists.
We also prove that any algorithm that decides whether a connected strongly-proportional allocation exists must make at least $n(n-1)/2$ queries, giving an asymptotically tight bound (within a factor of $2$) of $\Theta(n^2)$.
For agents with possibly unequal entitlements, we show that a lower bound number of queries to decide whether a connected strongly-proportional allocation exists is $n \cdot 2^{n-2}$.
Together with a result from \Cref{sec:general} later on the upper bound number of queries, this yields a tight bound of $\Theta(n \cdot 2^n)$ queries.

In \Cref{sec:general}, we consider agents who are not necessarily hungry.
The characterization from \Cref{sec:hungry} for hungry agents with equal entitlements does not work for non-hungry agents, which motivates us to find another characterization by considering permutations of agents.
We show that a connected strongly-proportional allocation exists if and only if there exists a permutation of agents such that when the agents go in the order as prescribed by the permutation and make their rightmost marks worth their entitlements to each of them one after another, the mark made by the last agent does not reach the end of the cake.
This result holds regardless of the agents' entitlements.
While an algorithm to determine this condition requires $n \cdot n!$ queries, we show that this number can be reduced by a factor of $2^{\omega(n)}$ to $n \cdot 2^{n-1}$ via dynamic programming.
We also prove a lower bound number of queries of $\Omega(n \cdot 2^n)$ to determine this condition, even for agents with equal entitlements.
Therefore, for agents who are not necessarily hungry, we also obtain a tight bound of $\Theta(n \cdot 2^n)$, whether the entitlements are equal or not.
A connected strongly-proportional allocation can be computed using $O(n \cdot 2^n)$ queries if it exists.

\Cref{tab:summary} summarizes of our results from \Cref{sec:hungry,sec:general}.

\begin{table}[ht]
\centering
\begin{tabular}{c|cc}
    & \textbf{hungry agents} & \textbf{general agents} \\ 
    \hline \rule{0cm}{3ex}
    \textbf{equal entitlements} 
        & $\Theta(n^2)$ (Thm \ref{thm:hungry_equal_tight}) 
        & $\Theta(n \cdot 2^n)$ (Thm \ref{thm:general_unequal_tight}) \\ 
        \rule{0cm}{4ex}
    \centered{\textbf{possibly unequal} \\ \textbf{entitlements}} 
        & $\Theta(n \cdot 2^n)$ (Thm \ref{thm:hungry_unequal_tight}) 
        & $\Theta(n \cdot 2^n)$ (Thm \ref{thm:general_unequal_tight})
\end{tabular}
\caption{Number of queries required to decide the existence of a connected strongly-proportional allocation of a cake for $n$ agents, and to compute one if it exists.}
\label{tab:summary}
\end{table}

In \Cref{sec:stronger}, we consider a stronger fairness notion where each agent $i$ needs to receive a connected piece of cake that is worth more than $w_i + z$ for some small $z$, where $w_i$ is agent $i$'s entitlement.
We show that the number of queries needed to decide whether such an allocation exists is in $\Theta(n \cdot 2^n)$, even for hungry agents with equal entitlements.
This is analogous to the results in \Cref{sec:hungry_unequal,sec:general}, which shows that the stronger fairness notion considered in this section does not make the problem any harder (nor easier).

In \Cref{sec:pies}, we consider a connected strongly-proportional allocation of a \emph{pie} instead of a cake, and show that no finite algorithm can decide the existence of such an allocation even for hungry agents with equal entitlements, demonstrating the intractability of the problem in this new setting.

\subsection{Further Related Work}

A weaker fairness notion of \emph{proportionality} is well-studied in cake-cutting literature.
It is known that a connected proportional allocation always exists for agents with equal entitlements and such an allocation can be computed using $\Theta(n \log n)$ queries \citep{steinhaus1948problem,even1984note,woeginger2007complexity}.
\citet{cseh2020complexity} presented an algorithm that finds a possibly non-connected proportional allocation for agents with general entitlements---in particular, their algorithm uses a finite but \emph{unbounded} number of queries when agents have irrational entitlements.
In contrast, we show that a connected \emph{strongly-proportional} allocation may not exist, and such an allocation can be computed (if it exists) using $\Theta(n \cdot 2^n)$ queries.
A number of works studied the number of \emph{cuts} required for a proportional allocation, rather than the number of queries \citep{segal2019cake,crew2020disproportionate}.

A parallel line of work studied a stronger fairness notion of \emph{super envy-freeness}: it requires, in addition to strong-proportionality, that each agent values the piece of every other agent at strictly less than $1/n$ the total cake value \citep{barbanel1996super,webb1999algorithm,cheze2020envy}.

\section{Preliminaries}
\label{sec:model}

Let the cake be denoted by $C = [0, 1]$.
The cake is to be allocated to a set of agents denoted by $[n] := \{1, \ldots, n\}$.
A \emph{piece of cake} is a finite union of closed intervals of the cake.
An \emph{allocation} of $C$ is a partition of $C$ into $n$ pairwise-disjoint%
\footnote{As mentioned in \Cref{sec:results}, two pieces of cake are also considered disjoint if their intersection is a subset of the endpoints of their respective intervals.}
pieces of cake $(X_1, \ldots, X_n)$ such that $C = X_1 \sqcup \cdots \sqcup X_n$; $X_i$ is the piece allocated to agent $i$.
An allocation is \emph{connected} if $X_i$ is a single interval for each $i \in [n]$.

The preference of each agent $i$ is represented by a \emph{valuation function} $V_i$ such that $V_i(X)$ is the value of the piece $X \subseteq C$ to agent $i$.
Each valuation function $V_i$ is defined on the algebra over $C$ generated by all intervals of $C$, and is non-negative (i.e., $V_i(X) \geq 0$ for all $X \subseteq C$ in the algebra), finitely additive (i.e., $V_i(X \cup Y) = V_i(X) + V_i(Y)$ for all disjoint $X, Y \subseteq C$ in the algebra), and normalized to one (i.e., $V_i(C) = 1$).
We assume that $F_i(x) := V_i([0, x])$ is a continuous function on $C$, and hence $V_i(\{x\}) = 0$ for all $x \in C$.
Therefore, $F_i$ is a non-decreasing function on $C$ with $F_i(0) = 0$, $F_i(1) = 1$, and $V_i([x, y]) = F_i(y) - F_i(x)$.
An agent $i$ is \emph{hungry} if $V_i(X) > 0$ for all intervals $X \subseteq C$ with positive length; this is equivalent to the condition that $F_i$ is strictly increasing.

Each agent $i$ has an \emph{entitlement} $w_i > 0$ of the cake such that $\sum_{i \in [n]} w_i = 1$.
Let $\mathbf{w}$ denote $(w_1, \ldots, w_n)$.
We say that agents have \emph{equal entitlements} if $w_i = 1/n$ for all $i \in [n]$.
For each subset $N \subseteq [n]$ of agents, define $w_{N} = \sum_{i \in N} w_i$.
Note that $w_{\emptyset} = 0$ and $w_{[n]} = 1$.
We say that agents have \emph{generic entitlements} if $w_N \neq w_{N'}$ for all distinct $N, N' \subseteq [n]$.

A \emph{(cake-cutting) instance} consists of the set of agents, their valuation functions $(V_i)_{i \in [n]}$, and their entitlements $\mathbf{w}$.

Given an instance, an allocation $(X_1, \ldots, X_n)$ is \emph{proportional} (resp.~\emph{strongly-proportional}) if $V_i(X_i) \geq w_i$ (resp.~$V_i(X_i) > w_i$) for all $i \in [n]$.
For agents with equal entitlements, a proportional (resp.~strongly-proportional) allocation requires every agent to receive a piece of cake with value at least (resp.~greater than) $1/n$.

Algorithms can make eval and mark queries of each agent in the Robertson-Webb model.
More specifically, for each agent $i \in [n]$, value $r \in [0, 1]$, and points $x, y \in C$ with $x \leq y$, $\textsc{Eval}_i(x, y)$ returns $V_i([x, y])$, and $\textsc{Mark}_i(x, r)$ returns the \emph{rightmost} (largest) point $z \in C$ such that $V_i([x, z]) = r$ (such a point exists due to the continuity of the valuations); if $V_i([x, 1]) < r$, then $\textsc{Mark}_i(x, r)$ returns $\infty$.

For $i \in [n]$ and $r \in [0, 1]$, a point $x \in C$ is an \emph{$r$-mark} of agent~$i$ if $V_i([0, x]) = r$.
While the point returned by $\textsc{Mark}_i(0, r)$ is an $r$-mark of agent $i$, the converse is not true since $\textsc{Mark}_i(0, r)$ only returns the \emph{rightmost} $r$-mark of agent $i$.
However, when agent $i$ is \emph{hungry}, then the $r$-mark is unique, and the two notions coincide.
Let $\mathcal{T}$ denote the subset $\{1/n, 2/n, \ldots, (n-1)/n\}$ of $C$---we shall consider $r$-marks for $r \in \mathcal{T}$ in \Cref{sec:hungry_equal}.

\section{Hungry Agents}
\label{sec:hungry}

We begin with the simpler case where all agents are hungry.
We first state a result which finds a connected strongly-proportional allocation of a cake for hungry agents using a small number of queries when given a connected \emph{proportional} allocation in which one agent has a strongly-proportional piece.
The proof proceeds by slightly moving the boundary between two adjacent agents' pieces such that an agent $j$ who received exactly $w_j$ eventually gets a slightly larger piece.

\begin{lemma}
\label{lem:hungry_one_agent}
Let an instance with $n$ hungry agents be given.
Suppose that we are given a connected proportional allocation $(X_1, \ldots, X_n)$ such that $V_i(X_i) > w_i$ for some $i \in [n]$.
Then, there exists a connected strongly-proportional allocation, and such an allocation can be computed using $O(n)$ queries.
\end{lemma}

\begin{proof}
First, we find the values of $V_j(X_j)$ for all $j \in [n]$.
If $V_j(X_j) > w_j$ for all $j \in [n]$, then we are done.
Otherwise, there exist two distinct agents $i, j \in [n]$ with neighboring pieces such that $V_i(X_i) > w_i$ and $V_j(X_j) = w_j$.
By slightly moving the boundary between $X_i$ and $X_j$, we can get a new allocation in which agents $i$ and $j$ each receives a piece worth more than $w_i$ and $w_j$ respectively.
To formally describe the process of moving the boundary, we consider two complementary cases.

\textbf{Case 1: $X_i$ is to the left of $X_j$. } 
Denote $X_i = [z_1, z_2]$ and $X_j = [z_2, z_3]$.
Let $y = \textsc{Mark}_i(z_1, w_i)$; note that $y \in (z_1, z_2)$ since $V_i(X_i) > w_i$.
Let $y^*$ be the midpoint of $y$ and $z_2$.
Adjust the two agents' pieces such that agent $i$ now receives $[z_1, y^*]$ and agent $j$ now receives $[y^*, z_3]$; see \Cref{fig:hungry_equal_condition_agent_ij} for an illustration.

\begin{figure}[ht]
\centering
\begin{tikzpicture}
    \node[label=below:$z_1$] at (0,0) (w1) {};
    \node[label=below:$y$] at (2,0) (y) {};
    \node[label=below:$y^*$] at (2.5,0) (ys) {};
    \node[label=below:$z_2$] at (3.0,0) (w2) {};
    \node[label=below:$z_3$] at (5,0) (w3) {};
    
    \draw [{|[width=8]}-] (w1.center) -- (y.center);
    \draw [{|[width=8]}-] (y.center) -- (ys.center);
    \draw [{|[width=8]}-] (ys.center) -- (w2.center);
    \draw [{|[width=8]}-{|[width=8]}] (w2.center) -- (w3.center);
    
    \draw [decorate,decoration={brace}] ([yshift=5]w1.center) -- ([yshift=5]y.center) node[pos=0.2, above, black][yshift=2]{worth $w_i$ to agent $i$};
    \draw [decorate,decoration={brace}] ([yshift=5]w2.center) -- ([yshift=5]w3.center) node[pos=0.8, above, black][yshift=2]{worth $w_j$ to agent $j$};
    
    \draw [decorate,decoration={brace,mirror}] ([yshift=-21]w1.center) -- ([xshift=-2,yshift=-21]ys.center) node[pos=0.2, below, black][yshift=-2]{agent $i$'s new piece};
    \draw [decorate,decoration={brace,mirror}] ([xshift=2,yshift=-21]ys.center) -- ([yshift=-21]w3.center) node[pos=0.8, below, black][yshift=-2]{agent $j$'s new piece};
\end{tikzpicture}
\caption{Agent $i$'s and $j$'s new pieces in the proof of \Cref{lem:hungry_one_agent}.}
\label{fig:hungry_equal_condition_agent_ij}
\end{figure}

Since $[z_1, y^*] \supsetneq [z_1, y]$ and the latter is worth $w_i$ to hungry agent $i$, the new piece, $[z_1, y^*]$, is worth more than $w_i$ to agent $i$.
Likewise, since $[y^*, z_3] \supsetneq [z_2, z_3]$ and the latter is worth $w_j$ to hungry agent $j$, the new piece, $[y^*, z_3]$, is worth more than $w_j$ to agent $j$.

\textbf{Case 2: $X_i$ is to the right of $X_j$. }
Denote $X_j = [z_1, z_2]$ and $X_i = [z_2, z_3]$. 
Let $y = \textsc{Mark}_i(z_2, V_i(X_i)-w_i)$; note that $y \in (z_2, z_3)$ since $V_i(X_i)>w_i$.
Let $y^*$ be the midpoint of $z_2$ and $y$.
Adjust the two agents' pieces such that agent $j$ now receives $[z_1, y^*]$ and agent $i$ now receives $[y^*, z_3]$.

Since $[z_1, y^*] \supsetneq [z_1, z_2]$ and the latter is worth $w_j$ to hungry agent $j$, the new piece, $[z_1, y^*]$, is worth more than $w_j$ to agent $j$.
Likewise, since $[y^*, z_3] \supsetneq [y, z_3]$ and the latter is worth $w_i$ to hungry agent $i$ (due to additivity, we have $V_i(y,z_3) = V_i(z_2,z_3) - V_i(z_2,y) = w_i$), the new piece, $[y^*, z_3]$, is worth more than $w_i$ to agent $i$.

In both Case 1 and Case 2, only agent $i$'s and $j$'s pieces change; all of the other agents' pieces do not change.
All in all, one additional agent $j$ receives more than $w_j$ of the cake.
Proceeding this way at most $n-1$ times yields a connected strongly-proportional allocation.

Finding the values of all $V_j(X_j)$ at the beginning requires $n$ queries, while the adjustment of the boundaries between two agents' pieces requires a constant number of queries, so the total number of queries is in $O(n)$.
\end{proof}

We present the results separately for agents with equal entitlements and agents with possibly unequal entitlements.
For $n$ hungry agents with equal entitlements, we state in \Cref{sec:hungry_equal} a simple necessary and sufficient condition for the existence of a connected strongly-proportional allocation.
We provide an asymptotically tight bound of $\Theta(n^2)$ for the number of queries needed by an algorithm to determine the existence of such an allocation, as well as to compute one such allocation if it exists.
For agents with possibly unequal entitlements, we show in \Cref{sec:hungry_unequal} that a lower bound number of queries needed to decide the existence of a connected strongly-proportional allocation is in $\Omega(n \cdot 2^n)$.

\subsection{Equal Entitlements}
\label{sec:hungry_equal}

Recall that $\mathcal{T} = \{1/n, 2/n, \ldots, (n-1)/n\}$.
Our condition uses a particular set of $r$-marks: those with $r \in \mathcal{T}$.

\begin{theorem}
\label{thm:hungry_equal_condition}
Let an instance with $n$ hungry agents with equal entitlements be given.
Then, a connected strongly-proportional allocation exists if and only if there exist two distinct agents $i, j \in [n]$ and $r \in \mathcal{T}$ such that the $r$-mark of agent $i$ is different from the $r$-mark of agent $j$.
\end{theorem}

\begin{proof}
Since the agents are hungry, there is exactly one $r$-mark of agent $i$ for each $r \in [0, 1]$ and $i \in [n]$.

$(\Rightarrow)$ We prove the contraposition.
Suppose that for each $r \in \mathcal{T}$, every agent has the same $r$-mark.
Every agent also has the same $0$-mark of $0$ and the same $1$-mark of $1$.
For each $t \in \{0, \ldots, n\}$, denote the common $t/n$-mark by $z_t$.

Consider now any connected allocation, which is represented by $n-1$ cuts on the cake.
For each $t \in [n-1]$, denote the $t$-th cut from the left by $x_t$; also denote $x_0 = 0$ and $x_n = 1$.
Each agent receives a piece $[x_{t-1}, x_t]$ for some $t \in [n]$, and every such piece is allocated to some agent.

Since $x_0=z_0$ and $x_n=z_n$, there must be some $t\in[n]$ for which $x_{t-1}\geq z_{t-1}$ and $x_t\leq z_t$.
This means that the piece $[x_{t-1},x_t]$ is contained in the interval $[z_{t-1},z_t]$. 
Let $i$ denote the agent who receives the piece $[x_{t-1},x_t]$.
Then, agent $i$'s value for her piece is
\begin{align*}
    V_i([x_{t-1}, x_t]) &\leq V_i([z_{t-1}, z_t]) 
    = V_i([0, z_{t}]) - V_i([0, z_{t-1}]) 
    = t/n - (t-1)/n = 1/n, 
\end{align*}
so the allocation is not strongly-proportional.
This holds for any connected allocation; therefore, no connected strongly-proportional allocation exists.

$(\Leftarrow)$ Suppose that there exist two distinct agents $i, j \in [n]$ and $r \in \mathcal{T}$ such that the $r$-mark of agent $i$ is different from the $r$-mark of agent $j$.
We shall construct a connected strongly-proportional allocation by first constructing a connected \emph{proportional} allocation such that at least one agent receives a piece with value more than $1/n$, then use \Cref{lem:hungry_one_agent} to construct a strongly-proportional one.

Let $t \in [n-1]$ be the integer such that $r = t/n$.
Let $i_L$ be an agent with the leftmost (smallest) $r$-mark among all the agents, and $i_R$ be an agent with the rightmost (largest) $r$-mark among all the agents (if there are multiple agents with the same leftmost or rightmost $r$-mark, we can choose an agent arbitrarily in each case).
Denote the leftmost $r$-mark by $z_L$ and the rightmost $r$-mark by $z_R$.
Note that $z_L < z_R$, since there are agents with different $r$-marks.

Since there are $n$ agents, there are $n$ $r$-marks (possibly some of them are equal) in the interval $[z_L, z_R]$.
Let $x \in [z_L, z_R]$ be the $t$-th $r$-mark from the left.
Then, there exists a partition of the agents into two subsets $N_1$ and $N_2$ such that
\begin{itemize}
    \item $|N_1| = t$, and the $r$-mark of all agents in $N_1$ is at most $x$, and
    \item $|N_2| = n-t$, and the $r$-mark of all agents in $N_2$ is at least $x$.
\end{itemize}
Every agent in $N_1$ values $[0, x]$ at least $r$, and every agent in $N_2$ values $[x, 1]$ at least $1-r$; see \Cref{fig:hungry_equal_condition_r_marks} for an illustration.

\begin{figure}[ht]
\centering
\begin{tikzpicture}
    \node[label=below:$0$] at (0,0) (v0) {};
    \node[label=below:$1$] at (5,0) (v1) {};
    \node[label=below:$z_L$] at (2.4,0) (zL) {};
    \node at (2.5,0) (z1) {};
    \node at (2.7,0) (z2) {};
    \node at (2.9,0) (z3) {};
    \node at (3.0,0) (z4) {};
    \node at (3.2,0) (z5) {};
    \node[label=below:$z_R$] at (3.3,0) (zR) {};
    \node at (3.0,0.7) (x1) {};
    \node[label=below:$x$] at (3.0,-0.7) (x2) {};
    
    \draw [{|[width=8]}-] (v0.center) -- (zL.center);
    \draw [{|[width=8]}-] (zL.center) -- (z1.center);
    \draw [|-] (z1.center) -- (z2.center);
    \draw [|-] (z2.center) -- (z3.center);
    \draw [|-] (z3.center) -- (z4.center);
    \draw [-] (z4.center) -- (z5.center);
    \draw [|-] (z5.center) -- (zR.center);
    \draw [{|[width=8]}-{|[width=8]}] (zR.center) -- (v1.center);
    \draw [-] (x1.center) -- (x2.center);
    
    \draw [decorate,decoration={brace}] ([xshift=-2,yshift=5]zL.center) -- ([xshift=-1,yshift=5]z4.center) node[pos=-0.9, above, black]{agents in $N_1$};
    \draw [decorate,decoration={brace}] ([xshift=1,yshift=5]z4.center) -- ([xshift=2,yshift=5]zR.center) node[pos=3.6, above, black]{agents in $N_2$};
\end{tikzpicture}
\caption{The $r$-marks of all the agents in the proof of \Cref{thm:hungry_equal_condition}. The point $x$ is at one of the $r$-marks and divides agents into $N_1$ and $N_2$.}
\label{fig:hungry_equal_condition_r_marks}
\end{figure}

Next, we consider any connected proportional cake-cutting algorithm as a black box (e.g., last diminisher).
We apply the algorithm on $[0, x]$ and $N_1$ such that every agent in $N_1$ receives a connected piece with value at least $1/t$ of her value of $[0, x]$, and apply the algorithm on $[x, 1]$ and $N_2$ such that every agent in $N_2$ receives a connected piece with value at least $1/(n-t)$ of her value of $[x, 1]$.
We show that this allocation (of $C = [0, 1]$) is proportional.
For an agent in $N_1$, since she values $[0, x]$ at least $r = t/n$, the piece she receives has value at least $(1/t)r = 1/n$.
Likewise, for an agent in $N_2$, since she values $[x, 1]$ at least $1-r = (n-t)/n$, the piece she receives has value at least $(1/(n-t))(1-r) = 1/n$.

Now, we show that agent $i_L$ or $i_R$ (or both) receives a piece with value strictly more than $1/n$.
If $x = z_R$, then we claim that agent $i_L$ receives such a piece.
Since the $r$-mark of agent $i_L$ is at $z_L < x$, we have $i_L \in N_1$.
Since agent $i_L$ is hungry, the piece $[0, x]$ is worth more than $r$ to her, and so the piece she receives has value more than $(1/t)r = 1/n$.
Otherwise, $x < z_R$, and a similar argument shows that agent $i_R$ receives such a piece.

Having established a connected proportional allocation in which at least one agent receives more than $1/n$, we apply \Cref{lem:hungry_one_agent} to obtain a connected \emph{strongly-proportional} allocation.
\end{proof}

It is interesting to compare the condition in \Cref{thm:hungry_equal_condition} with the one for non-connected allocations.
In both cases, a disagreement between \emph{two} agents is sufficient for allocating \emph{all} $n$ agents more than their fair share.
However, in the non-connected case, the disagreement can be in an $r$-mark for any $r \in (0, 1)$ (see the discussion in \Cref{sec:intro}), whereas in the connected case, the disagreement should be in an $r$-mark for some $r \in \mathcal{T}$; the $r$-marks for other values of $r$ are completely irrelevant.

It is clear from \Cref{thm:hungry_equal_condition} that we can decide whether a connected strongly-proportional allocation exists for hungry agents with equal entitlements by checking the $t/n$-marks of all of the $n$ agents for all $t \in [n-1]$.
This is described in \Cref{alg:hungry_equal}.
The number of queries used in the algorithm is at most $n(n-1)$.

\begin{algorithm}[ht]
  \caption{Determining the existence of a connected strongly-proportional allocation for $n$ hungry agents with equal entitlements.}
  \label{alg:hungry_equal}
  \begin{algorithmic}[1]
    \For{$t = 1, \ldots, n-1$}
      \State $z_t \leftarrow \textsc{Mark}_1(0, t/n)$ \algorithmiccomment{agent $1$'s $t/n$-mark}
      \For{$i = 2, \ldots, n$}
        \State \algorithmicif \ $\textsc{Mark}_i(0, t/n) \neq z_t$ \algorithmicthen \ \Return true
      \EndFor
    \EndFor
    \State \Return false
  \end{algorithmic}
\end{algorithm}

\begin{theorem}
\label{thm:hungry_equal_ub}
\Cref{alg:hungry_equal} decides whether a connected strongly-proportional allocation exists for $n$ hungry agents with equal entitlements using at most $n(n-1)$ queries.
\end{theorem}

Next, we show an asymptotically tight lower bound for the number of queries required to decide the existence of such an allocation for hungry agents.
The idea behind the proof is that we must check the $t/n$-marks of all the agents and all $t \in [n-1]$; otherwise, we can craft two instances---one with the $t/n$-marks coinciding, and the other with some $t/n$-marks not coinciding---that are consistent with the information obtained by the algorithm and yet give opposite results.
Doing this check requires at least $n(n-1)/2$ queries, as each query provides information on at most two points.

\begin{theorem}
\label{thm:hungry_equal_lb}
Any algorithm that decides whether a connected strongly-proportional allocation exists for $n$ hungry agents with equal entitlements requires at least $n(n-1)/2$ queries.
\end{theorem}

\begin{proof}
Suppose by way of contradiction that some algorithm decides the existence of a connected strongly-proportional allocation for $n$ hungry agents with equal entitlements using fewer than $n(n-1)/2$ queries.
We assume that for all $i \in [n]$, $r \in [0, 1]$ and $x \in C$, $\textsc{Eval}_i(0, x)$ returns the value $x$ and $\textsc{Mark}_i(0, r)$ returns the point $r$.
We make the following adjustments to the algorithm: whenever the algorithm makes an $\textsc{Eval}_i(x, y)$ query, it is instead given the answers to $\textsc{Mark}_i(0, x) = x$ and $\textsc{Mark}_i(0, y) = y$, and whenever the algorithm makes a $\textsc{Mark}_i(x, r)$ query, it is instead given the answers to $\textsc{Mark}_i(0, x) = x$ and $\textsc{Mark}_i(0, x+r) = x+r$.\footnote{Assuming $x+r \leq 1$; otherwise, $\textsc{Mark}_i(0, x+r) = \infty$.}
This means that every query made by the algorithm provides the algorithm only with information on at most \emph{two} $r$-marks of some agent and no other information that cannot be deduced from these $r$-marks.
Note that the algorithm can still deduce the values of $\textsc{Eval}_i(x, y)$ and $\textsc{Mark}_i(x, r)$ by taking the difference between the two answers given, which means that the information provided to the algorithm after the adjustment is a superset of the information provided to the algorithm before the adjustment.

The answers given to the algorithm are consistent with the instance where every agent's valuation is uniformly distributed over the cake---in which case there is no connected strongly-proportional allocation of the cake by \Cref{thm:hungry_equal_condition}---and so the algorithm should output ``false''.
However, we shall now show that the information provided to the algorithm is also consistent with an instance with a connected strongly-proportional allocation.
This means that the algorithm is not able to differentiate between the two, resulting in a contradiction.

Since fewer than $n(n-1)/2$ queries were made by the algorithm, fewer than $n(n-1)$ $r$-marks (for $r \in (0, 1)$) of all the agents are known.
In particular, there exists an agent $i \in [n]$ such that fewer than $n-1$ $r$-marks of agent $i$ are known, and hence there exists $t \in [n-1]$ such that the $t/n$-mark of agent $i$ is not known.
We now modify agent $i$'s valuation function slightly from the uniform distribution.
Let $\epsilon \in (0, 1/n)$ be a number such that every known $r$-mark of agent $i$ is of distance more than $\epsilon$ from $t/n$.
Let the $t/n$-mark of agent $i$ to be at $t/n + \epsilon$.
Construct agent $i$'s valuation function such that its distribution between all known $r$-marks of agent $i$ (including the new $t/n$-mark) is uniform within the respective intervals---note that this construction is valid and unique since these known $r$-marks are strictly increasing in $r$.
Let the other agents' valuation functions be uniformly distributed on the whole cake.
Then, agent $i$'s $t/n$-mark is different from every other agents' $t/n$-mark.
By \Cref{thm:hungry_equal_condition}, this instance admits a connected strongly-proportional allocation of the cake, forming the desired contradiction.
\end{proof}

\Cref{thm:hungry_equal_ub,thm:hungry_equal_lb} show that the number of queries required to determine the existence of a connected strongly-proportional allocation for $n$ hungry agents with equal entitlements is in $\Theta(n^2)$.
The same can be said for \emph{computing} such an allocation---we can modify \Cref{alg:hungry_equal} using the details in the proof of \Cref{thm:hungry_equal_condition} to output a connected strongly-proportional allocation of the cake instead, if such an allocation exists.

\begin{theorem}
\label{thm:hungry_equal_tight}
The number of queries required to decide the existence of a connected strongly-proportional allocation for $n$ hungry agents with equal entitlements, or to compute such an allocation if it exists, is in $\Theta(n^2)$.
\end{theorem}

\subsection{Possibly Unequal Entitlements}
\label{sec:hungry_unequal}

We now consider hungry agents who may not necessarily have equal entitlements.
Since the entitlement of a subset of agents may not be a multiple of $1/n$, we cannot use the condition in \Cref{thm:hungry_equal_condition} which uses $r$-marks for $r \in \mathcal{T}$.
This requires us to devise a more general condition to determine the existence of a connected strongly-proportional allocation, which can be checked using $O(n \cdot 2^n)$ queries.
Since the condition also works for non-hungry agents, we defer the discussion to \Cref{sec:general_ub} (see \Cref{thm:general_condition,thm:general_improved_ub}).

We now show an asymptotically-tight \emph{lower bound} for the case when agents may have unequal entitlements.
We show an even stronger result: for every vector of \emph{generic entitlements}, the number of queries required to decide the existence of a connected strongly-proportional allocation is in $\Omega(n \cdot 2^n)$.
The proof uses an adversarial argument similar to the one in \Cref{thm:hungry_equal_lb}.

\begin{theorem}
\label{thm:hungry_unequal_lb}
Let $\mathbf{w}$ be any vector of generic entitlements.
Then, any algorithm that decides whether a connected strongly-proportional allocation exists for $n$ hungry agents with entitlements $\mathbf{w}$ requires at least $n \cdot 2^{n-2}$ queries.
\end{theorem}

\begin{proof}
Since the entitlements are generic, we can arrange the $2^n$ different subsets of agents in strictly increasing order of their entitlements, i.e., we label the subsets of $[n]$ as $N_1, \ldots, N_{2^n}$ such that $w_{N_1} < \cdots < w_{N_{2^n}}$.
Note that $N_1 = \emptyset$ and $N_{2^n} = [n]$, giving $w_{N_1} = 0$ and $w_{N_{2^n}} = 1$.

Let $d = \min_{k=1}^{2^n-1} (w_{N_{k+1}} - w_{N_k})$ be the smallest gap between entitlements of different agent subsets.
For each $k \in \{2, \ldots, 2^n-1\}$, define $I_k = [w_{N_k},~ w_{N_k}+d/2]$.
Note that, by the choice of $d$, all the $I_k$ are pairwise disjoint.

Suppose by way of contradiction that some algorithm decides the existence of a connected strongly-proportional allocation for $n$ hungry agents with generic entitlements using fewer than $n \cdot 2^{n-2}$ queries.
We follow the construction in the proof of \Cref{thm:hungry_equal_lb} where we modify the algorithm such that every query returns information on at most \emph{two} $r$-marks of some agent, and these information are consistent with the instance where every agent's valuation is uniformly distributed over the cake.
Therefore, the algorithm should output ``false''.
We shall now show that the information provided to the algorithm is also consistent with an instance with a connected strongly-proportional allocation.
This means that the algorithm is not able to differentiate between the two, resulting in a contradiction.

Since fewer than $n \cdot 2^{n-2}$ queries were made by the algorithm, there exists an agent $i \in [n]$ such that at most $2^{n-2} - 1$ queries about the $r$-marks of agent $i$ (for $r \in (0, 1)$) are made.
Since each query returns information on at most \emph{two} $r$-marks, at most $2^{n-1} - 2$ $r$-marks of agent $i$ are known.
There are $2^{n-1} - 1$ non-empty subsets $N_k$ of $[n]$ that do not contain agent $i$, so there exists $k \in \{2, \ldots, 2^n-1\}$ such that $i \notin N_k$ 
and no known $r$-mark of agent $i$ is in the interval $I_k$.
Let $w = w_{N_k}$.
Let the $w$-mark of agent $i$ be at $w + d/4$.
Construct agent $i$'s valuation function such that its distribution between all known $r$-marks of agent $i$ (including the new $w$-mark) is uniform within the respective intervals---note that this construction is valid and unique since these known $r$-marks are strictly increasing in $r$.
Let the other agents' valuation functions be uniformly distributed on the whole cake.

We show that a connected strongly-proportional allocation exists.
The leftmost pieces are allocated to agents in $N_k$ in any arbitrary order, where every agent $j \in N_k$ receives a piece of length $w_j$.
Agent $i$ receives the piece $[w, w+w_i]$.
Finally, the remaining cake is allocated to the remaining agents such that every agent $j$ receives a piece of length $w_j$.
Note that every agent $j \in [n] \setminus \{i\}$ receives a piece worth exactly $w_j$, since their valuation functions are uniform.
The value of $[w + d/4, w + w_i]$ is $w_i$ to agent $i$, so agent $i$'s piece $[w, w+w_i] \supsetneq [w+d/4, w+w_i]$ is worth more than $w_i$ to hungry agent $i$.
Therefore, the allocation is proportional (and clearly connected) with agent $i$ receiving a piece strictly greater than $w_i$.
By \Cref{lem:hungry_one_agent}, a connected strongly-proportional allocation of the cake exists, forming the desired contradiction.
\end{proof}

Using the results from \Cref{thm:hungry_unequal_lb} and from \Cref{thm:general_improved_ub} later, we get a tight bound for hungry agents with possibly unequal entitlements.

\begin{theorem}
\label{thm:hungry_unequal_tight}
The number of queries required to decide the existence of a connected strongly-proportional allocation for $n$ hungry agents, or to compute such an allocation if it exists, is in $\Theta(n \cdot 2^n)$.
\end{theorem}

The lower bound in \Cref{thm:hungry_unequal_lb} is derived from the number of different values of $w_{N_k}$.
In particular, a lower bound number of queries is
\begin{align}
\label{eq:hungry_unequal_lb}
    \frac{1}{2} \sum_{i=1}^n ~ |\{ w_N : \emptyset \neq N \subseteq [n], i \notin N \}|.
\end{align}

For \emph{generic} entitlements, each term in the sum equals $2^{n-1}-1$, so we get roughly the lower bound of $n \cdot 2^{n-2}$ in \Cref{thm:hungry_unequal_lb}.
In contrast, for \emph{equal} entitlements, each term in the sum equals $n-1$, so we get the lower bound of $n(n-1)/2$ in \Cref{thm:hungry_equal_lb}.

For entitlements that are neither generic nor equal, the resulting lower bound is between these two extremes.
It is an interesting open question to find an algorithm with a query complexity matching the lower bound in \eqref{eq:hungry_unequal_lb} in these intermediate cases.
The main difficulty in extending our algorithm for equal entitlements (\Cref{alg:hungry_equal}) to unequal entitlements is due to the step in \Cref{thm:hungry_equal_condition} where we used a black-box algorithm for \emph{proportional} cake-cutting (such as last diminisher) to divide a part of the cake among the agents in $N_1$ and the other part among the agents in $N_2$.
Such a black box algorithm does not exist for unequal entitlements, since a connected proportional allocation might not even exist for unequal entitlements in the first place.

\section{General Agents}
\label{sec:general}

We now consider the general case where agents need not be hungry.
Recall that the condition we developed in \Cref{thm:hungry_equal_condition} involves checking for the coincidence of $r$-marks of all the agents for $r \in \mathcal{T}$.
However, there are some difficulties in generalizing the condition for non-hungry agents, even for equal entitlements.
The proof of \Cref{thm:hungry_equal_condition} relies crucially on the fact that an $r$-mark of an agent is unique, which may not be true for non-hungry agents.
For each agent $i \in [n]$, $F_i(x) = V_i([0, x])$ is a continuous function with domain $C = [0, 1]$ and range $[0, 1]$.
For each $r \in [0, 1]$, the set of $r$-marks of agent $i$ is $F_i^{-1}(\{r\})$.
Since $\{r\}$ is a closed set and $F_i$ is continuous, $F_i^{-1}(\{r\})$ is a non-empty closed set.
If agent $i$ is not necessarily hungry, then the fact that $F_i$ is non-decreasing implies the set of $r$-marks of agent $i$ is thus a non-empty closed \emph{interval} (though possibly the singleton set $[x, x] = \{x\}$).

Another difficulty is that there may be different instances with the same $t/n$-marks but give different results regarding the existence of such an allocation.
We show this via the following two examples.

\begin{example}
\label{eg:three_agents}
Consider a cake-cutting instance for $n = 3$ agents with equal entitlements where the cake is made up of $11$ homogenous regions.
The following table shows the agents' valuations for each region.

\begin{center}
\begin{tabular}{|l|ccccccccccc|}
\hline
\textbf{Alice} & 9 & 0 & 0 & 0 & 9 & 0 & 0 & 0 & 0 & 0 & 9  \\
\hline
\textbf{Bob}   & 1 & 4 & \multicolumn{1}{c|}{4} & 3 & 1 & \multicolumn{1}{c|}{5} & 1 & 1 & 2 & 4 & 1 \\
\hline
\textbf{Chana} & 1 & \multicolumn{1}{c|}{8} & 2 & 2 & 1 & 1 & 1 & \multicolumn{1}{c|}{2} & 4 & 4 & 1 \\
\hline
\end{tabular}
\end{center}
All agents value\footnote{
  The value of the cake should technically be normalized to $1$, but this can be done by simply dividing every value by $27$.
  We use integers here and in all subsequent examples for simplicity.
}
the entire cake at $27$, so the $t/n$-marks are at values $9$ and $18$.
Alice has two intervals of $t/n$-marks---the two intervals of zeros.
Bob and Chana each has two $t/n$-marks that are single points, denoted by vertical lines---note that both Bob and Chana are hungry.
We show that no connected strongly-proportional allocation exists.

Suppose by way of contradiction that a connected strongly-proportional allocation exists.
Alice must receive a piece with value larger than $9$, so her piece must touch the middle $9$ as well as either the left $9$ or the right $9$. In the former case, the cake remaining for Bob and Chana is at most:
\begin{center}
\begin{tabular}{|l|ccccccc|}
\hline
\textbf{Bob}   & 1 & 5 & 1 & 1 & 2 & 4 & 1 \\
\hline
\textbf{Chana} & 1 & 1 & 1 & 2 & 4 & 4 & 1 \\
\hline
\end{tabular}
\end{center}
In the latter case, the remaining cake is at most:
\begin{center}
\begin{tabular}{|l|ccccc|}
\hline
\textbf{Bob}   & 1 & 4 & 4 & 3 & 1 \\
\hline
\textbf{Chana} & 1 & 8 & 2 & 2 & 1 \\
\hline
\end{tabular}
\end{center}
In both cases, no matter how the remaining cake is divided between Bob and Chana, at least one agent gets a piece of cake with value at most $9$, so no connected strongly-proportional allocation exists.
\end{example}

\begin{example}
\label{eg:three_agents_positive}
Consider the following instance modified from \Cref{eg:three_agents}.

\begin{center}
\begin{tabular}{|l|ccccccccccc|}
\hline
\textbf{Alice} & 9 & 0 & 0 & 0 & 9 & 0 & 0 & 0 & 0 & 0 & 9  \\
\hline
\textbf{Bob}   & 1 & 4 & \multicolumn{1}{c|}{4} & 3 & 1 & \multicolumn{1}{c|}{5} & \textcolor{red}{5} & \textcolor{red}{1} & \textcolor{red}{1} & \textcolor{red}{1} & 1 \\
\hline
\textbf{Chana} & 1 & \multicolumn{1}{c|}{8} & 2 & 2 & 1 & 1 & 1 & \multicolumn{1}{c|}{2} & 4 & 4 & 1 \\
\hline
\end{tabular}
\end{center}
The $t/n$-marks of the agents are identical to those in \Cref{eg:three_agents}. However, a connected strongly-proportional allocation exists, as the following table shows:

\begin{center}
\begin{tabular}{|l|ccccc|cc|cccc|}
\hline
\textbf{Alice} & 9 & 0 & 0 & 0 & 9 &  &  &  &  &  & \\
\hline
\textbf{Bob}   &  &  &  &  &  & 5 & 5 &  &  &  & \\
\hline
\textbf{Chana} &  &  &  &  &  &  &  & 2 & 4 & 4 & 1 \\
\hline
\end{tabular}
\end{center}
\end{example}

\Cref{eg:three_agents,eg:three_agents_positive} show that the condition for determining the existence of a connected strongly-proportional allocation cannot be extended trivially from the result for hungry agents.
Instead, let us discuss the extent to which the results from \Cref{sec:hungry_equal} can be extended.
We start with a necessary condition regarding the $r$-marks for $r \in \mathcal{T}$.
This condition is similar to that in \Cref{thm:hungry_equal_condition}.

\begin{proposition}[Necessary condition]
\label{prop:necessary}
Let an instance with $n$ agents with equal entitlements be given.
If there exists a connected strongly-proportional allocation, then there exist two distinct agents $i, j \in [n]$ and $r \in \mathcal{T}$ such that the interval of $r$-marks of agent $i$ is disjoint%
\footnote{Unlike for pieces of cake where ``disjoint'' means \emph{finite} intersection, we revert to the standard definition of ``disjoint'' to mean \emph{empty} intersection for intervals involving $r$-marks.}
from the interval of $r$-marks of agent $j$.
\end{proposition}

\begin{proof}
Let a connected strongly-proportional allocation be given, and let $\sigma : [n] \to [n]$ be the permutation such that agent $\sigma(k)$ receives the $k$-th piece from the left.
Suppose by way of contradiction that for each $r \in \mathcal{T}$, the intervals of $r$-marks of every pair of agents have non-empty intersection.
We show by backward induction that for each $k \in [n]$, every agent in $\{\sigma(1), \ldots, \sigma(k)\}$ assigns a total value of at most $k/n$ to the leftmost $k$ pieces.

The base case of $k = n$ is clear---every agent assigns a value of at most $n/n = 1$ to the whole cake.
Suppose that the statement is true for $k+1$ for some $k \in [n-1]$; we shall prove the statement for $k$.
Since agent $\sigma(k+1)$ assigns a value of at most $(k+1)/n$ to the leftmost $k+1$ pieces, the left endpoint of her piece must be strictly to the left of her interval of $k/n$-marks in order for her to receive a piece worth more than $1/n$.
Now, consider agent $\sigma(i)$ for $i \in [k]$.
Since agent $\sigma(i)$'s interval of $k/n$-marks intersects with agent $\sigma(k+1)$'s interval of $k/n$-marks, the remaining cake after removing $\sigma(k+1)$'s piece is worth at most $k/n$ to agent $\sigma(i)$.
This proves the inductive statement.

Now, the statement for $k = 1$ states that agent $\sigma(1)$ receives a piece worth at most $1/n$. This contradicts the assumption that the allocation is strongly-proportional.
\end{proof}

Next, we provide a sufficient condition for a connected strongly-proportional allocation using intervals of $r$-marks for $r \in \mathcal{T}$.
It differs from the necessary condition of \Cref{prop:necessary} in that it requires the intervals of $r$-marks of \emph{all} agents, rather than just two, to be pairwise disjoint.

\begin{proposition}[Sufficient condition]
\label{prop:sufficient}
Let an instance with $n$ agents with equal entitlements be given.
If there exists $r \in \mathcal{T}$ such that the intervals of $r$-marks of all agents are pairwise disjoint, then a connected strongly-proportional allocation exists.
\end{proposition}

\begin{proof}
Let $r = t/n$ be such that the intervals of $r$-marks of all the agents are pairwise disjoint.
Since these intervals of $r$-marks are closed and are pairwise disjoint, we can arrange them from smallest to largest.
Moreover, the gap between any two consecutive intervals of $r$-marks is a non-empty open interval.
Let $N_1$ be the set of $t$ agents whose intervals of $r$-marks are the smallest, and let $N_2$ be the remaining agents.
Then, there exists a point $x$ between the $r$-marks of the agents in $N_1$ and that of the agents in $N_2$.
Note that the $t$ agents in $N_1$ each values the cake $[0, x]$ more than $t/n$, and the $n-t$ agents in $N_2$ each values the cake $[x, 1]$ more than $(n-t)/n$.
We apply any connected proportional cake-cutting algorithm on each of $[0, x]$ on $N_1$ and $[x, 1]$ on $N_2$ such that every agent receives a connected piece worth more than $1/n$.
This gives a connected strongly-proportional allocation.
\end{proof}

\Cref{prop:necessary,prop:sufficient} coincide for $n = 2$ agents, yielding the following result.

\begin{corollary}
\label{cor:two_general_agents}
Let an instance with two agents with equal entitlements be given.
Then, a connected strongly-proportional allocation exists if and only if the intervals of $1/2$-marks of the two agents are disjoint.
\end{corollary}

The two conditions do not coincide for $n \geq 3$ agents, however.
In the search for a necessary and sufficient condition for three or more agents, one could consider weakening the condition in \Cref{prop:sufficient} to require the intervals of $r$-marks of just \emph{two} agents to be pairwise disjoint for some $r \in \mathcal{T}$.
However, as one could see in \Cref{eg:three_agents}, even when the intervals of $r$-marks of Bob and Chana are disjoint for \emph{all} $r \in \mathcal{T}$, there is no connected strongly-proportional allocation.
Another possibility is to require that the interval of $r$-marks of one agent to be disjoint from every other agents' intervals of $r$-marks for some $r \in \mathcal{T}$.
However, the following example shows that this is still not correct.

\begin{example}
\label{eg:three_agents2}
Consider the following instance for $n = 3$ agents.

\begin{center}
\begin{tabular}{|l|ccccc|}
\hline
\textbf{Alice} & 4 & 2 & \multicolumn{1}{c|}{2} & 1 & 3 \\
\hline
\textbf{Bob}   & 4 & 0 & 2 & \multicolumn{1}{c|}{2} & 4 \\
\hline
\textbf{Chana} & 4 & 0 & 2 & \multicolumn{1}{c|}{2} & 4 \\
\hline
\end{tabular}
\end{center}
All agents value the entire cake at $12$, so the $2/3$-marks are at value $8$, denoted by vertical lines.
Alice's $2/3$-mark is disjoint from Bob's and Chana's $2/3$-marks.
However, no connected strongly-proportional allocation exists.
If Alice receives the leftmost piece or the rightmost piece, then the remaining cake is worth at most $8$ to both Bob and Chana, and both of them cannot simultaneously get a piece worth more than $4$ each since they have identical valuations.
If Alice receives the middle piece instead, then Bob and Chana must receive the leftmost and the rightmost piece in some order.
However, the leftmost piece must touch the third region, the rightmost piece must touch the fourth region, and Alice's piece is confined to the third and fourth regions which is only worth at most $3$ to her.
\end{example}

\Cref{eg:three_agents,eg:three_agents_positive,eg:three_agents2} show that the existence of a connected strongly-proportional allocation cannot be determined based on $t/n$-marks alone.
This inspires us to find another condition that characterizes the existence of a connected strongly-proportional allocation.

In \Cref{sec:general_ub}, we generalize the condition from \Cref{thm:hungry_equal_condition} for $n$ non-hungry agents, regardless of whether they have equal entitlements or not.
We show that this condition can be checked by an algorithm using $O(n \cdot 2^n)$ queries.
Now, the result in \Cref{thm:hungry_unequal_lb} says that the lower bound number of queries needed for an algorithm to determine the existence of a connected strongly-proportional allocation for $n$ hungry agents with generic entitlements is $\Omega(n \cdot 2^n)$---we show in \Cref{sec:general_lb} that this lower bound also applies to (not necessarily hungry) agents with \emph{equal entitlements}.

\subsection{Upper Bound}
\label{sec:general_ub}

Our condition requires agents to mark pieces of cake one after another in a certain order.
We explain this operation more precisely.
Let $\sigma : [n] \to [n]$ be a permutation of agents, and let $x \in C$ and $r_1, \ldots, r_n \in [0, 1]$.
The agents proceed in the order $\sigma(1), \ldots, \sigma(n)$.
Agent $\sigma(1)$ starts first and makes a mark at $x_1 = \textsc{Mark}_{\sigma(1)}(x, r_{\sigma(1)})$, the rightmost point such that $[x, x_1]$ is worth $r_{\sigma(1)}$ to her.
Then, agent $\sigma(2)$ continues from $x_1$, and makes a mark at $x_2 = \textsc{Mark}_{\sigma(2)}(x_1, r_{\sigma(2)})$, the rightmost point such that $[x_1, x_2]$ is worth $r_{\sigma(2)}$ to her.
Each agent $\sigma(i)$ repeats the same process of making a mark at $x_i = \textsc{Mark}_{\sigma(i)}(x_{i-1}, r_{\sigma(i)})$ such that $[x_{i-1}, x_i]$ is the largest possible piece worth $r_{\sigma(i)}$ to her.
We shall overload the definition of \textsc{Mark} and define%
\footnote{The subscript of \textsc{Mark} here is a permutation $\sigma$, not an agent number.} 
$\textsc{Mark}_{\sigma}(x, \mathbf{r})$ as the point $x_n$ resulting from this sequential marking process, where $\mathbf{r} = (r_1, \ldots, r_n)$.
If $[x_{i-1}, 1]$ is worth less than $r_{\sigma(i)}$ to agent $\sigma(i)$ at any point, then $\textsc{Mark}_{\sigma}(x, \mathbf{r})$ is defined as $\infty$.
This operation is described in \Cref{alg:mark_sigma}.
Note that each $\textsc{Mark}_{\sigma}(x, \mathbf{r})$ operation requires at most $n$ ($\textsc{Mark}_i$) queries.

\begin{algorithm}[ht]
  \caption{Computing $\textsc{Mark}_{\sigma}(x, \mathbf{r})$ for $n$ agents.}
  \label{alg:mark_sigma}
  \begin{algorithmic}[1]
    \State $x_0 \leftarrow x$
    \For{$i = 1, \ldots, n$}
      \State $x_i \leftarrow \textsc{Mark}_{\sigma(i)}(x_{i-1}, r_{\sigma(i)})$
      \State \algorithmicif \ $x_i = \infty$ \algorithmicthen \ \Return $\infty$
    \EndFor
    \State \Return $x_n$
  \end{algorithmic}
\end{algorithm}

Our necessary and sufficient condition for $n$ (possibly non-hungry) agents requires us to check whether the point $\textsc{Mark}_{\sigma}(0, \mathbf{w})$ is less than 1 for some permutation $\sigma$.
The point $\textsc{Mark}_{\sigma}(0, \mathbf{w})$ is determined when agents go in the order as prescribed by $\sigma$ and make their rightmost marks worth their entitlements to each of them one after another.
The idea behind the proof is that starting from the agent who receives the rightmost piece in $\sigma$ and going leftwards, each agent is able to move the boundaries of her piece such that she receives a small piece of cake with positive value $\epsilon$ from the right and gives away a small piece of cake with value $\epsilon/2$ to the agent on the left, thereby increasing the value of her piece by a positive value $\epsilon/2$.

\begin{theorem}
\label{thm:general_condition}
Let an instance with $n$ agents be given.
Then, a connected strongly-proportional allocation exists if and only if there exists a permutation $\sigma : [n] \to [n]$ such that $\textsc{Mark}_{\sigma}(0, \mathbf{w}) < 1$.
\end{theorem}

\begin{proof}
$(\Rightarrow)$ Suppose that a connected strongly-proportional allocation exists.
Let $\sigma : [n] \to [n]$ be the permutation such that agent $\sigma(k)$ receives the $k$-th piece from the left in this allocation, and let $y_0, y_1, \ldots, y_n$ be the points such that agent $\sigma(k)$ receives the piece $[y_{k-1}, y_k]$ with $y_0 = 0$ and $y_n = 1$.
We shall show that $\textsc{Mark}_{\sigma}(0, \mathbf{w}) < 1$.
Let $x_0, x_1, \ldots, x_n$ be the points as described by \Cref{alg:mark_sigma} for $\textsc{Mark}_{\sigma}(0, \mathbf{w})$.
We shall show by induction that $x_k < y_k$ for all $k \in [n]$; then, $\textsc{Mark}_{\sigma}(0, \mathbf{w}) = x_n < y_n = 1$ gives the desired conclusion.

For the base case $k = 1$, we have $x_1 = \textsc{Mark}_{\sigma(1)}(0, w_{\sigma(1)})$, so $x_1$ is a point for which $V_{\sigma(1)}([0,x_1]) = w_{\sigma(1)}$.
Since agent $\sigma(1)$ receives a piece $[y_0, y_1] = [0, y_1]$ worth more than $w_{\sigma(1)}$, we must have $x_1 < y_1$.
For the inductive case, assume that $x_k < y_k$ for some $k \in [n-1]$, and consider $k+1$.
We have $x_{k+1} = \textsc{Mark}_{\sigma(k+1)}(x_k, w_{\sigma(k+1)}) \leq \textsc{Mark}_{\sigma(k+1)}(y_k, w_{\sigma(k+1)})$ since $x_k < y_k$.
Since agent $\sigma(k+1)$ receives a piece $[y_k, y_{k+1}]$ worth more than $w_{\sigma(k+1)}$, we have $\textsc{Mark}_{\sigma(k+1)}(y_k, w_{\sigma(k+1)}) < y_{k+1}$.
Therefore, the result $x_{k+1} < y_{k+1}$ holds, proving the induction statement.

$(\Leftarrow)$ Suppose that there exists a permutation $\sigma : [n] \to [n]$ such that $\textsc{Mark}_{\sigma}(0, \mathbf{w}) < 1$.
Let $x_0, x_1, \ldots, x_n$ be the points as described by \Cref{alg:mark_sigma} for $\textsc{Mark}_{\sigma}(0, \mathbf{w})$.
Since $x_k$, which is $\textsc{Mark}_{\sigma(k)}(x_{k-1}, w_{\sigma(k)})$, is the rightmost point $z$ such that $[x_{k-1}, z]$ is worth $w_{\sigma(k)}$ to agent $\sigma(k)$, the piece $[x_{k-1}, y_k]$ is worth more than $w_{\sigma(k)}$ to agent $\sigma(k)$ whenever $y_k > x_k$.

We shall define the points $y_1, \ldots, y_n \in C$ in the reverse order such that $y_k > x_k$ for all $k \in [n]$.
Define $y_n = 1 > \textsc{Mark}_{\sigma}(0, \mathbf{w}) = x_n$.
Next, for each $k \in [n-1]$, assume that $y_{k+1}$ is defined such that $y_{k+1} > x_{k+1}$.
Since $[x_k, y_{k+1}]$ is worth more than $w_{\sigma(k+1)}$ to agent $\sigma(k+1)$, it must be worth $w_{\sigma(k+1)} + \epsilon_{k+1}$ to agent $\sigma(k+1)$ for some $\epsilon_{k+1} > 0$.
Define $y_k = \textsc{Mark}_{\sigma(k+1)}(x_k, \epsilon_{k+1}/2)$.
Then, we have $y_k > x_k$.
This completes the definition of $y_1, \ldots, y_n$.

Let $y_0 = x_0 = 0$.
We shall show that the allocation with the cut points at $y_0, \ldots, y_n$ such that $[y_{k-1}, y_k]$ is allocated to agent $\sigma(k)$ for $k \in [n]$ is strongly-proportional.
Agent $\sigma(1)$ receives $[y_0, y_1] = [x_0, y_1]$ which is worth more than $w_{\sigma(1)}$ to her.
For $k \in \{2, \ldots, n\}$, since $[x_{k-1}, y_k]$ is worth $w_{\sigma(k)} + \epsilon_k$ and $[x_{k-1}, y_{k-1}]$ is worth $\epsilon_k/2$ to agent $\sigma(k)$, the piece $[y_{k-1}, y_k]$ is worth $(w_{\sigma(k)} + \epsilon_k) - \epsilon_k/2 > w_{\sigma(k)}$ to agent $\sigma(k)$.
This completes the proof.
\end{proof}

The condition in \Cref{thm:general_condition} reduces to the condition in \Cref{thm:hungry_equal_condition} for hungry agents with equal entitlements, i.e., when $\mathbf{w} = (1/n, \ldots, 1/n)$.
In particular, when every agent has the same $r$-mark for each $r \in \mathcal{T}$, then each of the $n$ marks made in the $\textsc{Mark}_{\sigma}(0, \mathbf{w})$ operation coincides at some $x_i \in \mathcal{T} \cup \{1\}$ for every permutation, and so $\textsc{Mark}_{\sigma}(0, \mathbf{w}) = 1$ for all $\sigma$.
This corresponds to the case where no connected strongly-proportional allocation exists.

The analysis in \Cref{thm:general_condition} relies crucially on the fact that the $\textsc{Mark}_i$ queries return the \emph{rightmost} points.
If the \emph{leftmost} points are returned instead, then the condition does not work---this can be seen from \Cref{eg:three_agents} where Chana, Alice, and Bob could (left-)mark their respective $1/n$ piece of the cake one after another in this order and still have a positive-valued cake left, but no connected strongly-proportional allocation exists as we demonstrated in \Cref{eg:three_agents}.

\begin{center}
\begin{tabular}{|l|cc|ccc|cccc|cc|}
\hline
\textbf{Alice} &   &   & 0 & 0 & 9 &   &   &   &   & \textcolor{red}{0} & \textcolor{red}{9}  \\
\hline
\textbf{Bob}   &   &   &   &   &   & 5 & 1 & 1 & 2 & \textcolor{red}{4} & \textcolor{red}{1} \\
\hline
\textbf{Chana} & 1 & 8 &   &   &   &   &   &   &   & \textcolor{red}{4} & \textcolor{red}{1} \\
\hline
\end{tabular}
\end{center}

We can determine whether the condition in \Cref{thm:general_condition} holds by checking all permutation $\sigma$ to see whether the point $\textsc{Mark}_{\sigma}(0, \mathbf{w})$ is less than $1$ for some $\sigma$.
Since there are $n!$ possible permutations of $[n]$ and each $\textsc{Mark}_{\sigma}$ operation requires at most $n$ queries, the total number of queries required in the algorithm is at most $n \cdot n!$.

However, we can reduce the number of queries to $n \cdot 2^{n-1}$ by dynamic programming.
Our approach is similar to the method used in \citet{aumann2012computing}---in their work, they iteratively find a value $k$ such that there exists a connected allocation where every agent receives \emph{at least} $k$, while here we require every agent $i$ to receive a connected piece with value \emph{strictly more} than $w_i$.

We now describe our algorithm.
For every subset $N \subseteq [n]$, our algorithm caches the \emph{best} mark $b_N$ obtained by the subset of agents.
The best mark $b_N$ is the leftmost point possible over all permutations of the agents in $N$ when the agents go in the order as prescribed by the permutation and make their rightmost marks worth their entitlements to each of them one after another.
The algorithm aims to compute this point for every $N$.

The best mark for the empty set of agents is initialized as $b_{\emptyset} = 0$.
Thereafter, for every $k \in [n]$, we assume that the best mark for every subset of $k-1$ agents is calculated earlier and cached.
We now need to find $b_N$ for every subset $N \subseteq [n]$ with $k$ agents.
The last agent to make the best mark for $N$ could be any of the agents $i \in N$.
Therefore, for each $i \in N$, we retrieve the best mark for $N \setminus \{i\}$, which is $b_{N \setminus \{i\}}$ and has been cached earlier, and let agent $i$ make the rightmost mark such that the cake starting from $b_{N \setminus \{i\}}$ is worth $w_i$ to her.
By iterating through all $i \in N$, we find the leftmost such point and cache this point as $b_N$.
When $k = n$, we obtain $b_{[n]}$, which is the best $\textsc{Mark}_{\sigma}(0, \mathbf{w})$ over all permutations $\sigma$.
Therefore, the algorithm returns ``true'' if $b_{[n]} < 1$, and ``false'' otherwise.
This implementation reduces the number of queries by a factor of $2^{\omega(n)}$.

This algorithm is described in \Cref{alg:general_improved}.
The correctness of the algorithm relies on the statement in \Cref{thm:general_condition} and the fact that $b_{[n]}$ in the algorithm is less than $1$ if and only if there exists a permutation $\sigma : [n] \to [n]$ such that $\textsc{Mark}_{\sigma}(0, \mathbf{w}) < 1$.

\begin{algorithm}[ht]
  \caption{Determining the existence of a connected strongly-proportional allocation for $n$ agents with fewer queries.}
  \label{alg:general_improved}
  \begin{algorithmic}[1]
    \State $b_{\emptyset} \leftarrow 0$
    \For{$k = 1, \ldots, n$}
      \For{each subset $N \subseteq [n]$ with $|N| = k$}
        \State $b_N \leftarrow \infty$
        \For{each agent $i \in N$}
          \State $y \leftarrow \textsc{Mark}_i(b_{N \setminus \{i\}}, w_i)$ \label{ln:wi}
          \State \algorithmicif \ $y < b_N$ \algorithmicthen \ $b_N \leftarrow y$ \algorithmiccomment{this finds the ``best'' $b_N$}
        \EndFor
      \EndFor
    \EndFor
    \State \algorithmicif \ $b_{[n]} < 1$ \algorithmicthen \ \Return true \algorithmicelse \ \Return false
  \end{algorithmic}
\end{algorithm}

\begin{theorem}
\label{thm:general_improved_ub}
\Cref{alg:general_improved} decides whether a connected strongly-proportional allocation exists for $n$ agents using at most $n \cdot 2^{n-1}$ queries.
\end{theorem}

\begin{proof}
To show that \Cref{alg:general_improved} is correct, it suffices to show that $b_{[n]}$ in the algorithm is less than $1$ if and only if there exists a permutation $\sigma : [n] \to [n]$ such that $\textsc{Mark}_{\sigma}(0, \mathbf{w}) < 1$, by \Cref{thm:general_condition}.

$(\Rightarrow)$ If $b_{[n]}$ in the algorithm is less than $1$, then $b_{[n]}$ is contributed by some agent $i_n \in [n]$ making the rightmost $w_{i_n}$-mark after $b_{[n] \setminus \{i_n\}}$.
Let $\sigma(n) = i_n$.
We then consider the agent $i_{n-1}$ contributing the rightmost $w_{i_{n-1}}$-mark for $b_{[n] \setminus \{i_n\}}$, and so on.
Repeat the procedure $n-1$ times to obtain the identities of the agents $\sigma(n-1), \ldots, \sigma(1)$.
Then, $\sigma$ is the desired permutation.

$(\Leftarrow)$ Suppose there exists a permutation $\sigma : [n] \to [n]$ such that $\textsc{Mark}_{\sigma}(0, \mathbf{w}) < 1$.
For each $k \in [n]$, let $N_k = \{\sigma(1), \ldots, \sigma(k)\}$, and let $x_k^\sigma$ be the mark where agents $\sigma(1), \ldots, \sigma(k)$ make their rightmost mark worth their entitlements to each of them one after another in this order.
We prove by induction on $k$ that $b_{N_k} \leq x_k^\sigma$.
The base case of $k = 1$ is clear, as the two quantities are equal.
Assume that the inequality is true for $k \in [n-1]$; we shall prove the result for $k+1$.
The point $b_{N_{k+1}}$ is the smallest point over all permutations where agents $\sigma(1), \ldots, \sigma(k+1)$ make their rightmost $1/n$-mark one after another in some order.
In particular, $x_{k+1}^\sigma$ is one of these points under consideration.
Therefore, we must have $b_{N_{k+1}} \leq x_{k+1}^\sigma$, proving the induction statement.
Then, we have $b_{[n]} = b_{N_n} \leq x_n^\sigma = \textsc{Mark}_{\sigma}(0, \mathbf{w}) < 1$.

Next, we show that the number of queries made by \Cref{alg:general_improved} is at most $n \cdot 2^{n-1}$.
Let $k \in [n]$ be given.
There are $\binom{n}{k}$ subsets $N$ with cardinality $k$, and for each $N$, each of the $|N| = k$ agents makes a mark query.
This means that $k \binom{n}{k}$ queries are made.
Hence, the total number of queries is $\sum_{k=1}^n k \binom{n}{k} = n \cdot 2^{n-1}$ by a combinatorial identity.
\end{proof}

\subsection{Lower Bound}
\label{sec:general_lb}

\Cref{thm:hungry_unequal_lb} provides a lower bound for hungry agents with unequal entitlements; we shall now prove a similar lower bound for general agents with equal entitlements.

At a high level, the technique used is similar to that in the proofs of \Cref{thm:hungry_equal_lb,thm:hungry_unequal_lb}: we use an adversarial argument where we construct an instance with agents having uniform valuations on the cake such that no strongly-proportional allocation exists, but tweak the valuations slightly depending on the queries made.
However, the details from the proof of \Cref{thm:hungry_equal_lb} cannot be used directly since the existence of a connected strongly-proportional allocation is not solely dependent on the $r$-marks for $r \in \mathcal{T}$ for non-hungry agents (see the discussion at the beginning of \Cref{sec:general}), and the details from the proof of \Cref{thm:hungry_unequal_lb} cannot be used directly since \Cref{thm:hungry_unequal_lb} requires the entitlements to be generic.

Instead, we construct the following instance with $n \geq 3$ agents.
The cake is divided into $2n-1$ parts.
The odd parts (i.e., the 1st, 3rd, \ldots, $(2n-1)$-th parts) are non-valuable to agents $1$ to $n-1$, and worth $1/n$ each to agent $n$.
The even parts (i.e., the 2nd, 4th, \ldots, $(2n-2)$-th parts) are valuable to agents $1$ to $n-1$, and non-valuable to agent $n$.
For $i \in [n-1]$, agent $i$'s first $n-2$ valuable parts (i.e., the 2nd, 4th, \ldots, $(2n-4)$-th parts) are worth $a_i/(n-2)$ each to agent $i$ for some carefully selected $a_i$, and the last valuable part (i.e., the $(2n-2)$-th part) is worth $1-a_i$ to agent $i$.
See \Cref{fig:general_equal_lb_cake} for an illustration.

\begin{figure}[ht]
\centering
\resizebox{\textwidth}{!}{
\begin{tabular}{|c|cc;{2pt/2pt}c;{2pt/2pt}cc;{2pt/2pt}ccc|}
\hline
\textbf{Agent $1$} & $0$ & $a_1/(n-2)$ & \multirow{5}{*}{\centered{$\cdots$ \\ (total: $n-2$ \\ identical copies)}} & $0$ & $a_1/(n-2)$ & $0$ & $1-a_1$ & $0$ \\
$\vdots$ & $\vdots$ & $\vdots$ &  & $\vdots$ & $\vdots$ & $\vdots$ & $\vdots$ & $\vdots$ \\
\textbf{Agent $n-1$} & $0$ & $a_{n-1}/(n-2)$ &  & $0$ & $a_{n-1}/(n-2)$ & $0$ & $1-a_{n-1}$ & $0$ \\
\textbf{Agent $n$} & $1/n$ & $0$ &  & $1/n$ & $0$ & $1/n$ & $0$ & $1/n$ \\
\hline
\end{tabular}
}
\caption{Construction of the cake used in the proof of \Cref{thm:general_equal_lb}.}
\label{fig:general_equal_lb_cake}
\end{figure}

Consider a connected strongly-proportional allocation with equal entitlements.
Agent $n$'s piece has to include pieces from at least two consecutive odd parts in order for her value to be greater than $1/n$.
By a clever choice of $a_i$ for $i \in [n-1]$, we force these two odd parts to be the \emph{rightmost} odd parts.
This leaves the remaining $2n-4$ parts for agents $1$ to $n-1$.
Removing all the non-valuable parts for these agents, the remaining valuable parts of the cake are worth $a_i$ to agent $i \in [n-1]$.
Divide all valuations and entitlements by $a_i$ for each $i \in [n-1]$.
Then, this is equivalent to a cake with value $1$ to every agent such that each agent's entitlement is $w'_i = 1/na_i$.
If we select the $a_i$'s carefully such that $\sum_{i \in [n-1]} w'_i = 1$ and the entitlements $w'_i$'s are \emph{generic}, then we can invoke \Cref{thm:hungry_unequal_lb} to show that the lower bound number of queries is in $\Omega(n \cdot 2^n)$.

\begin{theorem}
\label{thm:general_equal_lb}
Any algorithm that decides whether a connected strongly-proportional allocation exists for $n$ agents with equal entitlements requires $\Omega(n \cdot 2^n)$ queries.
\end{theorem}

\begin{proof}
Let $M$ be a sufficiently large constant (particularly, $M \geq 2^n n^2$), and for each $i \in [n-1]$, define $w'_i = \frac{M+2^{i-1}}{(n-1)M+2^{n-1}-1}$ and $a_i = \frac{1}{nw'_i}$.
Note that $\sum_{i \in [n-1]} w'_i = 1$.

Consider a cake with $2n-1$ parts as illustrated in \Cref{fig:general_equal_lb_cake}.
We now show that the valuations on the cake are valid.
It suffices to show that for each $i \in [n-1]$, each of $a_i/(n-2)$ and $1-a_i$ is positive.
It is clear that $w'_i$ is positive, which means that $a_i$ and hence $a_i/(n-2)$ are positive.
We have $1-a_i = \frac{M-2^{n-1}+2^{i-1}n+1}{n(M+2^{i-1})}$ which is positive since $M > 2^{n-1}$.
These show that the valuations are valid.

We show that any algorithm that makes fewer than $(n-1)(2^{n-3}-n+3) \in \Omega(n \cdot 2^n)$ queries may not be able to decide whether a connected strongly-proportional allocation exists.
The proof idea is similar to that in the proofs of \Cref{thm:hungry_equal_lb,thm:hungry_unequal_lb}.
We assume that the answer to every query made by the algorithm is consistent with the instance where the valuation of each agent is uniformly distributed in their valuable parts, which are the even parts for agents $i \in [n-1]$ and the odd parts for agent $n$.
We show that regardless of what the algorithm outputs as its answer, there are instances which contradict the answer.

\textbf{Case 1: The algorithm outputs ``true''. }
Consider the instance where the valuation of each agent is uniformly distributed in their valuable parts.
We show that a connected strongly-proportional allocation cannot exist in this instance.

Suppose on the contrary that a connected strongly-proportional allocation exists.
Recall that agents have equal entitlements, which means that every agent receives a piece worth more than $1/n$.
Agent $n$'s piece has to include pieces from at least two consecutive odd parts in order for her value to be greater than $1/n$, which means that agent $n$'s piece has to contain at least one of the valuable parts of agent $i \in [n-1]$ completely.

We now show that for each $i \in [n-1]$, we have $a_i/(n-2) > 1-a_i$.
We have
\begin{align*}
    \frac{a_i}{n-2}-(1-a_i) = \frac{M-2^{i-1}n^2+n(2^{n-1}+2^i-1)+1}{(n-2)n(M+2^{i-1})} > 0,
\end{align*}
where the inequality holds because $M > 2^{i-1}n^2$.
This shows that  $a_i/(n-2) > 1-a_i$, which implies that each of the left valuable parts is worth more than the rightmost valuable part for agent $i \in [n-1]$.

For $i \in [n-1]$, since $a_i/(n-2) > 1-a_i$, no matter which valuable part(s) of agent $i$ is given to agent $n$, each of the remaining $\leq n-2$ valuable parts is worth at most $a_i/(n-2)$ to agent $i$.
Since the valuations are uniformly distributed within each valuable part, agent $i$ receives more than $(1/n) \div (a_i/(n-2)) = (n-2)w'_i$ of a valuable part.
Therefore, agents $1$ to $n-1$ receive more than 
\begin{align*}
    \sum_{i=1}^{n-1} (n-2)w'_i = (n-2) \sum_{i=1}^{n-1} w'_i = n-2
\end{align*}
valuable parts in total.
This is not possible, and hence, no connected strongly-proportional allocation exists.

\textbf{Case 2: The algorithm outputs ``false''. }
We shall now show by construction that the information provided to the algorithm is also consistent with an instance with a connected strongly-proportional allocation, resulting in a contradiction.

Let agent $n$ receive the rightmost two consecutive odd parts, so that agent $n$ receives more than $1/n$.
This leaves the remaining $2n-4$ parts for agents $1$ to $n-1$.
Removing all the non-valuable parts for all agents $i \in [n-1]$, the remaining valuable parts of the cake are worth $a_i$ to agent $i$.
Divide all valuations and entitlements by $a_i$ for each $i \in [n-1]$---note that this does not change the existence of a connected strongly-proportional allocation.
Then, this is equivalent to a cake with value $1$ to every agent in $[n-1]$ such that agent $i$'s entitlement is $1/na_i = w'_i$. 
Note that $\sum_{i \in [n-1]} w'_i = 1$, so we have reduced the problem to finding a connected strongly-proportional allocation on a modified instance with $n-1$ hungry agents such that agent $i \in [n-1]$ has an entitlement of $w'_i$.

We claim that the entitlements are generic.
To see this, let $N, N' \subseteq [n-1]$ such that $\sum_{i \in N} w'_i = \sum_{i \in N'} w'_i$.
Since the denominators of the $w'_i$'s are equal to each other, we have $\sum_{i \in N} (M+2^{i-1}) = \sum_{i \in N'} (M+2^{i-1})$.
Since $M$ is larger than $\sum_{i \in [n-1]} 2^{i-1}$, we must have $|N| = |N'|$, which implies that $\sum_{i \in N} 2^{i-1} = \sum_{i \in N'} 2^{i-1}$.
The only way this is possible is when $N = N'$, which proves that the entitlements are generic.

Since the $n-1$ agents are hungry in this modified instance, we can use the construction in the proof of \Cref{thm:hungry_unequal_lb} for agents $1, \ldots, n-1$, which shows that, with fewer than $(n-1)2^{n-3}$ queries, the answers are consistent with the existence of a connected strongly-proportional allocation.
However, note that the marks of agent $i \in [n-1]$ between every $a_i/(n-2)$ part in \Cref{fig:general_equal_lb_cake} are already known, which translate to the $t/(n-2)$-marks for $t \in \{1, \ldots, n-3\}$.
This means that a total of $(n-1)(n-3)$ marks are known, which requires at most the same number of queries.
Therefore, with fewer than $(n-1)2^{n-3} - (n-1)(n-3) \in \Omega(n \cdot 2^n)$ queries, there exists an instance consistent with the information provided by the queries that admits a connected strongly-proportional allocation.
This contradicts the output of the algorithm.
\end{proof}

The upper bound from \Cref{thm:general_improved_ub} and the lower bound from \Cref{thm:general_equal_lb} imply that the number of queries required to determine the existence of a connected strongly-proportional allocation is in $\Theta(n \cdot 2^n)$, even for agents with equal entitlements.
The same tight bound also holds for \emph{computing} such an allocation if it exists---this can be shown by modifying \Cref{alg:general_improved} slightly by following the details in the second half of the proof of \Cref{thm:general_condition}.

\begin{theorem}
\label{thm:general_unequal_tight}
The number of queries required to decide the existence of a connected strongly-proportional allocation for $n$ agents, or to compute such an allocation if it exists, is in $\Theta(n \cdot 2^n)$, even for agents with equal entitlements.
\end{theorem}

\section{Stronger than Strongly-Proportional}
\label{sec:stronger}

We have so far only considered allocations which are \emph{strongly-proportional}---agents receive pieces with value strictly more than their entitlements.
Strong proportionality does not guarantee that agents receive pieces beyond just a small crumb more than their proportional piece.
It would indeed be useful if we can guarantee that agents receive a fixed positive amount more than their entitlements.
This motivates us to consider an even stronger fairness notion: given some fixed value $z > 0$, can each agent $i$ receive a piece with value more than $w_i + z$?

It is easy to adapt \Cref{alg:general_improved} to this setting by replacing $w_i$ in \Cref{ln:wi} of the algorithm with $w_i + z$---this gives an upper bound number of queries to determine the existence of such an allocation, or to compute such an allocation if it exists.

\begin{theorem}
\label{thm:stronger_ub}
Let $\mathbf{w}$ be any vector of entitlements.
Then, for any positive constant $z$, there exists an algorithm that decides whether a connected allocation exists for $n$ agents in which each agent $i$ receives a piece with value more than $w_i + z$ using at most $n \cdot 2^{n-1}$ queries.
\end{theorem}

We now show a matching lower bound, even for hungry agents with equal entitlements.
The proof is similar in spirit to the proof of \Cref{thm:general_equal_lb}.

\begin{theorem}
\label{thm:stronger_lb}
Let $n$ be given.
Then, for any positive constant $z < \frac{1}{n(n-1)}$, any algorithm that decides whether a connected allocation exists for $n$ hungry agents in which each agent receives a piece with value more than $1/n + z$ requires $\Omega(n \cdot 2^n)$ queries.
\end{theorem}

\begin{proof}
Let $n \geq 3$ and $z \in (0, \frac{1}{n(n-1)})$ be given.
Define $\epsilon$ such that
\begin{align*}
    \epsilon = \min \left\{ \frac{1}{n(n-1)} - z, \; \frac{nz}{n-1} \right\};
\end{align*}
note that $\epsilon > 0$.
Let $M$ be a sufficiently large constant (to be decided later), and for each $i \in [n-1]$, define $w'_i = \frac{M+2^{i-1}}{(n-1)M+2^{n-1}-1}$ and $a_i = \frac{1/n+z}{w'_i}$.
Note that we have $\lim_{M \to \infty} w'_i = \frac{1}{n-1}$; therefore, we choose a value $M \geq 2^n$ such that $\frac{1}{n-1} - \epsilon < w'_i < \frac{1}{n-1} + \epsilon$ for all $i \in [n-1]$.
Note also that $\sum_{i \in [n-1]} w'_i = 1$.

Consider a cake with two parts---the \emph{left} part and the \emph{right} part.
The left part is worth $a_i$ to agent $i \in [n-1]$ and $1 - 1/n - z$ to agent $n$.
The right part is worth $1-a_i$ to agent $i \in [n-1]$ and $1/n + z$ to agent $n$.
See \Cref{fig:stronger_lb_cake} for an illustration.

\begin{figure}[ht]
\centering
\begin{tabular}{|c|cc|}
\hline
\textbf{Agent $1$} & $a_1$ & $1-a_1$ \\
$\vdots$ & $\vdots$ & $\vdots$ \\
\textbf{Agent $n-1$} & $a_{n-1}$ & $1-a_{n-1}$ \\
\textbf{Agent $n$} & $1-1/n-z$ & $1/n+z$ \\ 
\hline
\end{tabular}
\caption{Construction of the cake used in the proof of \Cref{thm:stronger_lb}.}
\label{fig:stronger_lb_cake}
\end{figure}

We now show that the valuations on the cake are valid.
This is clear for agent $n$'s valuation since $0 < 1/n + z < 1$.
It suffices to show that for $i \in [n-1]$, each of $a_i$ and $1-a_i$ is less than $1$.
We have
\begin{align*}
    a_i &= \frac{\frac{1}{n}+z}{w'_i} \\
    &< \frac{\frac{1}{n}+z}{\frac{1}{n-1}-\epsilon} \tag*{(since $w'_i > \frac{1}{n-1}-\epsilon$)} \\
    &\leq \frac{\frac{1}{n} + \frac{1}{n(n-1)} - \epsilon}{\frac{1}{n-1}-\epsilon} \tag*{(since $\epsilon \leq \frac{1}{n(n-1)} - z$)} \\
    &= \frac{\frac{1}{n-1}-\epsilon}{\frac{1}{n-1}-\epsilon} \\
    &= 1.
\end{align*}
Now, instead of showing $1-a_i < 1$, we show an even stronger statement of $1-a_i < 1/n+z$.
We have
\begin{align*}
    1-a_i &= 1 - \frac{\frac{1}{n}+z}{w'_i} \\
    &< 1 - \frac{\frac{1}{n}+z}{\frac{1}{n-1}+\epsilon} \tag*{(since $w'_i < \frac{1}{n-1}+\epsilon$)} \\
    &= \frac{1+(n-1)\epsilon}{1+(n-1)\epsilon} - \frac{\frac{n-1}{n}+(n-1)z}{1+(n-1)\epsilon} \\
    &= \frac{\frac{1}{n} + (n-1)\epsilon - (n-1)z}{1+(n-1)\epsilon}
\end{align*}
and
\begin{align*}
    \frac{1}{n} + z = \frac{(1+(n-1)\epsilon)(\frac{1}{n} + z)}{1+(n-1)\epsilon} = \frac{\frac{1}{n} +(n-1)(\frac{1}{n}+z)\epsilon + z}{1+(n-1)\epsilon}.
\end{align*}
Taking the difference between the two expressions, we have
\begin{align*}
    (1-a_i) - \left( \frac{1}{n} + z \right) &< \frac{(n-1)(1-\frac{1}{n}-z)\epsilon - nz}{1+(n-1)\epsilon} \\
    &< \frac{(n-1)\epsilon - nz}{1+(n-1)\epsilon} \tag*{(since $1 - \frac{1}{n} - z < 1$)} \\
    &\leq \frac{nz - nz}{1+(n-1)\epsilon} \tag*{(since $\epsilon \leq \frac{nz}{n-1}$)} \\
    &= 0,
\end{align*}
showing indeed that $1-a_i < 1/n+z$.
These show that the valuations are valid.

We show that any algorithm that makes fewer than $(n-1)(2^{n-3}-1) \in \Omega(n \cdot 2^n)$ queries may not be able to decide whether a connected strongly-proportional allocation exists.
The proof idea is similar to that in the proofs of \Cref{thm:hungry_equal_lb,thm:hungry_unequal_lb}.
We assume that the answer to every query made by the algorithm is consistent with the instance where the valuation of each agent is uniformly distributed in each of the left and the right parts.
We show that regardless of what the algorithm outputs as its answer, there are instances which contradict the answer.

\textbf{Case 1: The algorithm outputs ``true''. }
Consider the instance where the valuation of each agent is uniformly distributed in each part.
Note that all agents are hungry.
We show that a connected allocation in which each agent receives a piece with value more than $1/n + z$ cannot exist in this instance.

Suppose on the contrary that such an allocation exists.
Assume first that agent $n$ receives the rightmost piece.
Since the right part of the cake is worth $1/n + z$ to agent $n$, agent $n$'s piece has to contain the whole of the right part.
The remaining $n-1$ agents' pieces are contained in the left part.
For $i \in [n-1]$, agent $i$ receives more than $(1/n + z)/a_i = w'_i$ of the left part.
Therefore, agents $1$ to $n-1$ receive more than $\sum_{i \in [n-1]} w'_i = 1$ of the left part in total.
This is not possible.

Therefore, some agent $j \in [n-1]$ receives the rightmost piece.
Since $1 - a_j < 1/n + z$, agent $j$'s piece has to contain the whole of the right part.
The remaining $n-1$ agents' pieces are contained in the left part.
For $i \in [n-1] \setminus \{j\}$, agent $i$ receives more than $(1/n + z)/a_i = w'_i$ of the left part.
Agent $n$ receives more than 
\begin{align*}
    \frac{1/n + z}{1 - 1/n - z} > \frac{1/n + z}{1 - 1/n} = \frac{1}{n-1} + \frac{nz}{n-1} \geq \frac{1}{n-1} + \epsilon > w'_j
\end{align*}
of the left part, where the second inequality holds since $\epsilon \leq \frac{nz}{n-1}$.
Therefore, agents in $[n] \setminus \{j\}$ receive more than $\sum_{i \in [n-1]} w'_i = 1$ of the left part in total.
This is not possible, and hence, no such allocation exists.

\textbf{Case 2: The algorithm outputs ``false''. }
We shall now show by construction that the information provided to the algorithm is also consistent with an instance with a connected allocation in which each agent receives a piece with value more than $1/n + z$, resulting in a contradiction.

Let agent $n$ receive the right part of the cake, so that agent $n$ receives a piece with value $1/n + z$.
While the value of this piece is not more than $1/n + z$ yet, we will fix this later.
Divide all valuations of the left part of the cake and entitlements by $a_i$ for each $i \in [n-1]$---note that this does not change the existence of such a connected allocation.
Then, this is equivalent to a cake with value $1$ to every agent $i \in [n-1]$ such that agent $i$ needs to receive a piece of cake with value more than $\frac{1/n+z}{a_i} = w'_i$.
Note that $\sum_{i \in [n-1]} w'_i = 1$, so we have reduced the problem to finding a connected allocation on a modified instance with $n-1$ hungry agents such that agent $i \in [n-1]$ has an entitlement of $w'_i$.

We claim that the entitlements are generic.
To see this, let $N, N' \subseteq [n-1]$ such that $\sum_{i \in N} w'_i = \sum_{i \in N'} w'_i$.
Since the denominators of the $w'_i$'s are equal to each other, we have $\sum_{i \in N} (M+2^{i-1}) = \sum_{i \in N'} (M+2^{i-1})$.
Since $M$ is larger than $\sum_{i \in [n-1]} 2^{i-1}$, we must have $|N| = |N'|$, which implies that $\sum_{i \in N} 2^{i-1} = \sum_{i \in N'} 2^{i-1}$.
The only way this is possible is when $N = N'$, which proves that the entitlements are generic.

Since the $n-1$ agents are hungry in this modified instance, we can use the construction in the proof of \Cref{thm:hungry_unequal_lb} for agents $1, \ldots, n-1$, which shows that, with fewer than $(n-1)2^{n-3}$ queries, the answers are consistent with the existence of a connected strongly-proportional allocation.
Note that the marks between the left and the right part of the cake are known.
This means that a total of $n-1$ marks are known, which requires at most the same number of queries.
Therefore, with fewer than $(n-1)2^{n-3} - (n-1) \in \Omega(n \cdot 2^n)$ queries, there exists an instance consistent with the information provided by the queries that admits a connected strongly-proportional allocation for agents in $[n-1]$.

We now have a connected proportional allocation such that each agent in $[n-1]$ receives a piece with value more than $1/n + z$ and agent $n$ receives a piece with value exactly $1/n + z$.
By a similar proof as that in \Cref{lem:hungry_one_agent}, we can slightly move the boundary of agent $n$'s piece such that every agent receives a piece with value more than $1/n + z$.
This contradicts the output of the algorithm.
\end{proof}

\Cref{thm:stronger_lb} shows that no algorithm can decide whether there exists a connected allocation in which the \emph{egalitarian value} is more than $1/n+z$---the egalitarian value is the smallest value of an agent's piece in an allocation.
\citet{aumann2012computing} prove a closely-related result, but in a different computational model.
They assume that the agents' valuations are piecewise-constant and given explicitly to the algorithm.
In this model, they prove that it is NP-hard to approximate the optimal egalitarian value to a factor better than 2.
Specifically, they show a reduction from an instance of the NP-hard problem 3-dimensional matching (3DM) to a cake-cutting instance.
They show that if the answer to the 3DM instance is ``yes'', then there exists an allocation where each agent gets value at least $k$; and if the answer is ``no'', then every allocation gives some agent value at most $k/2$. In their reduction, $k$ is at least $4/n$.
In contrast, we provide an unconditional exponential lower bound in the (harder) query model.
Also, our result holds for a different range of possible $k = 1/n + z$ values.

\Cref{thm:stronger_ub,thm:stronger_lb} give a tight bound to the number of queries, even for hungry agents with equal entitlements.

\begin{theorem}
\label{thm:stronger_tight}
Let $n$ be given.
Then, for any positive $z < \frac{1}{n(n-1)}$, the number of queries required to decide the existence of a connected allocation for $n$ agents in which every agent receives a piece with value more than $w_i + z$, or to compute such an allocation if it exists, is in $\Theta(n \cdot 2^n)$, even for hungry agents with equal entitlements.
\end{theorem}

Our results complete the picture on the query complexity of computing a connected allocation that guarantees each agent a piece with a certain value.
For $n$ hungry agents, the query complexity of computing a connected allocation in which each agent receives a piece with value at least $1/n$, more than $1/n$, and more than $1/n + z$ for some small $z > 0$ is $\Theta(n \log n)$, $\Theta(n^2)$, and $\Theta(n \cdot 2^n)$ respectively, if such an allocation exists.

\section{Pies}
\label{sec:pies}

We now consider a \emph{pie}, where the resource is modeled by a circle instead of by an interval.
We use $C = [0, 1)$ to represent the pie, but in contrast to the cake version, the endpoints $0$ and $1$ are ``joined'' together---they are considered the same point.
Therefore, the piece $[0, a] \cup [b, 1)$ is also considered a connected piece for any $a, b \in C$ with $a \leq b$.
Several papers have studied the special properties of pie-cutting \citep{stromquist2007pie,thomson2007children,brams2008proportional,barbanel2009cutting}.

We show that the problem of deciding the existence of a connected strongly-proportional allocation of a pie is intractable.

\begin{theorem}
\label{thm:pie}
No finite algorithm can decide the existence of a connected strongly-proportional allocation of a pie, even for hungry agents with equal entitlements.
\end{theorem}

\begin{proof}
Suppose by way of contradiction that some finite algorithm decides the existence of a connected strongly-proportional allocation of a pie for $n$ agents with equal entitlements.
Assume that the information provided to the algorithm by eval and mark queries is consistent with that where every agent's valuation is uniformly distributed over the pie (in which case there is no connected strongly-proportional allocation of the pie), and so the algorithm should output ``false''.
However, we shall now show that the information provided to the algorithm is also consistent with an instance with a connected strongly-proportional allocation.
This means that the algorithm is not able to differentiate between the two, resulting in a contradiction.

Let $P$ be the set of all points on the pie mentioned by the algorithm or by the queries---for example, if an $\textsc{Eval}_i(x, y)$ query is made by the algorithm, or if a $\textsc{Mark}_i(x, r)$ query is made by the algorithm and $y$ is returned, then $x$ and $y$ are added to $P$.
Since the algorithm is finite, $P$ is finite.
For $x \in C$, define $\bar{x} = \{x, x+1/n, \ldots, x+(n-1)/n\} \subseteq C$ where all numbers in the set are modulo $1$.
(From now on, every point mentioned is modulo $1$.)
Since $P$ is finite, there exists a point $x \in C$ such that $\bar{x} \cap P = \emptyset$.
Fix $x$.
Let $\epsilon \in (0, 1/n)$ be a number smaller than the distance between any element in $\bar{x}$ and any element in $P$.

Construct agent $1$'s valuation function such that $V_1([x, p]) = p-x$ for all $p \in P$, $V_1([x, x+1/n+\epsilon]) = 1/n$, and its distribution between every two adjacent points in $P \cup \{x+1/n+\epsilon\}$ is uniform within the respective intervals based on these valuations---note that this construction is valid and unique since these known $r$-marks (starting at the point $x$) are strictly increasing in $r$.
All other $n-1$ agents have valuation functions uniformly distributed over the pie.
Note that all agents are hungry.
By changing the axis to start at the point $x$, we see that the $1/n$-mark (starting at the point $x$) of agent $1$ is at $x+1/n+\epsilon$ while that of the other agents are at $x+1/n$.
Therefore, the $1/n$-mark of agent $1$ is different from that of the other agents.
By \Cref{thm:hungry_equal_condition}, there exists a connected strongly-proportional allocation starting from the point $x$.
This means that the algorithm is not able to differentiate between the two instances.
\end{proof}

\section{Conclusion}

We have studied necessary and sufficient conditions for the existence of a connected strongly-proportional allocation on the interval cake (\Cref{thm:hungry_equal_condition,thm:general_condition}).
We have shown that computing this condition requires $\Theta(n \cdot 2^n)$ queries even for agents with equal entitlements (\Cref{thm:general_unequal_tight}) or hungry agents with generic entitlements (\Cref{thm:hungry_unequal_tight}), and $\Theta(n^2)$ for hungry agents with equal entitlements (\Cref{thm:hungry_equal_tight}).
The same bounds hold for the computation of such an allocation if it exists.
We have also shown that for connected allocations where each agent receives a small value $z$ more than their proportional share, the number of queries to decide the existence of such allocations is in $\Theta(n \cdot 2^n)$ (\Cref{thm:stronger_tight}).
Finally, we have shown that no finite algorithm can decide the existence of a connected strongly-proportional allocation of a pie (\Cref{thm:pie}).

A natural question that arose from our work is whether there is an algorithm that (asymptotically) attains the lower bound in \eqref{eq:hungry_unequal_lb} for hungry agents with entitlements that are neither generic nor equal.

Additionally, our work can be extended in the following ways:

\begin{itemize}
\item \textbf{Chores.} \emph{Chore-cutting} is a variant of cake-cutting in which agents have \emph{negative} valuations for every piece of the cake.
\item \textbf{Beyond the unit interval.} We can consider cakes with more complex topologies, such as graphical cakes \citep{bei2021dividing}, tangled cakes \citep{igarashi2023fair}, and two-dimensional cakes \citep{segal2017fair}.
\item \textbf{Envy-freeness.} It is known that, in every cake-cutting instance, a connected envy-free allocation exists \citep{stromquist1980how,su1999rental}.
What conditions are necessary and sufficient for the existence of a connected strongly-proportional allocation that is also envy-free?
\end{itemize}
We may also consider a weaker fairness notion of \emph{proportionality} instead---we give a brief discussion in \Cref{ap:proportionality}.

\section*{Acknowledgements}
This work was partially supported by the Hungarian Scientific Research Fund (OTKA grant K143858), by the Deutsche Forschungsgemeinschaft (DFG, German Research Foundation) under grant number 513023562, by NKFIH OTKA-129211, by the Israel Science Foundation grant 712/20, and by the Singapore Ministry of Education under grant number MOE-T2EP20221-0001.
We thank Tassle and Neal Young 
for their helpful answers,%
\footnote{https://cstheory.stackexchange.com/q/53901/9453}
Warut Suksompong for his feedback on an earlier draft, and participants of the COMSOC online seminar and the anonymous ECAI 2024 reviewers for their comments.
We are grateful to an anonymous ECAI 2024 reviewer for suggesting the solution to Case 2 in the proof of \Cref{lem:hungry_one_agent}.

\bibliographystyle{plainnat}
\bibliography{main}

\clearpage
\appendix

\section{Left-Marks and Right-Marks}
\label{ap:left_right_marks}

Our results assume that algorithms have access to eval and right-mark queries in the Robertson-Webb model.
We show that our results also hold if algorithms have access to eval and \emph{left-mark} queries.

For each agent $i \in [n]$, value $r \in [0, 1]$, and point $x \in C$, define \emph{left-mark} such that \emph{$\textsc{Left-Mark}_i(x, r)$} is the \emph{leftmost} (smallest) point $z \in C$ such that $V_i([x, z]) = r$ (such a point exists due to the continuity of the valuations); if $V_i([x, 1]) < r$, then $\textsc{Left-Mark}_i(x, r)$ returns $\infty$.
We define \emph{$\textsc{Right-Mark}_i(x, r)$} to be the same as $\textsc{Mark}_i(x, r)$ in the main text, but use ``right-mark'' in this section to avoid confusion with left-mark.
Note that for a hungry agent $i$, $\textsc{Left-Mark}_i(x, r) = \textsc{Right-Mark}_i(x, r)$ for all $x \in C, r \in [0, 1]$.

Let $\mathcal{I}$ be an instance with $n$ agents with valuation functions $(V_i)_{i \in [n]}$ and entitlements $\mathbf{w}$.
A \emph{mirrored instance} $\cev{\mathcal{I}}$ of $\mathcal{I}$ is the instance with $n$ agents with valuation functions $(\cev{V}_i)_{i \in [n]}$ and entitlements $\mathbf{w}$ such that $\cev{V}_i([x, y]) = V_i([1-y, 1-x])$ for all $x, y \in C$ with $x \leq y$.
In other words, the cake in $\cev{\mathcal{I}}$ is ``mirrored'' from the cake in $\mathcal{I}$.
A class of instances $\mathcal{C}$ is \emph{closed under mirror} if $\mathcal{I} \in \mathcal{C}$ implies that $\cev{\mathcal{I}} \in \mathcal{C}$.

Note that the classes of instances we consider in our results are closed under mirror.
We now show that if there is an algorithm, having access to right-mark queries, that can determine the existence of a connected strongly-proportional allocation in a class of instances closed under mirror, then there exists another algorithm, having access to left-mark queries, that can also do the same in the same class of instances.
Furthermore, the number of queries made by the new algorithm is within a factor of $2$ from the number of queries made by the old algorithm.
This shows that our results hold in whichever model of Robertson-Webb, and the use of right-mark is only for convenience.

\begin{proposition}
Let $\mathcal{C}$ be a class of instances closed under mirror.
Suppose there exists an algorithm $\mathcal{A}$ such that for each instance in $\mathcal{C}$, algorithm $\mathcal{A}$ can determine the existence of a connected strongly-proportional allocation in the instance using at most $k$ eval and right-mark queries and using no left-mark queries.
Then, there exists an algorithm $\mathcal{B}$ such that for each instance in $\mathcal{C}$, algorithm $\mathcal{B}$ can determine the existence of a connected strongly-proportional allocation in the instance using at most $2k$ eval and left-mark queries and using no right-mark queries.
\end{proposition}

\begin{proof}
Let $\mathcal{I}$ be an instance in $\mathcal{C}$ in which we wish to determine the existence of a connected strongly-proportional allocation using eval and left-mark queries and using no right-mark queries.
By assumption, there exists an algorithm $\mathcal{A}$ that can determine the existence of a connected strongly-proportional allocation in $\cev{\mathcal{I}}$ using at most $k$ eval and right-mark queries and using no left-mark queries.

Our algorithm $\mathcal{B}$ simulates algorithm $\mathcal{A}$ on $\mathcal{I}$ as follows:
\begin{itemize}
\item \textbf{Case 1: $\mathcal{A}$ makes an $\textsc{Eval}_i(x, y)$ query on $\cev{\mathcal{I}}$.}\\
Then, algorithm $\mathcal{B}$ makes an $\textsc{Eval}_i(1-y, 1-x)$ query on $\mathcal{I}$.
\item \textbf{Case 2: $\mathcal{A}$ makes a $\textsc{Right-Mark}_i(x, r)$ query on $\cev{\mathcal{I}}$.}\\
Then, algorithm $\mathcal{B}$ makes a $\textsc{Left-Mark}_i(0, \textsc{Eval}_i(0, 1-x)-r)$ query on $\mathcal{I}$.
\end{itemize}
Let us verify that the two implementations are identical.
\begin{itemize}
    \item Suppose $r = \textsc{Eval}_i(x, y)$ on $\cev{\mathcal{I}}$.
    Then, $\textsc{Eval}_i(1-y, 1-x)$ on $\mathcal{I}$ is equal to $V_i([1-y, 1-x])$, which is equal to $\cev{V}_i([x, y]) = r$.
    \item Suppose $y = \textsc{Right-Mark}_i(x, r)$ on $\cev{\mathcal{I}}$.
    Then, since the value of $\textsc{Eval}_i(0, 1-x)$ on $\mathcal{I}$ is equal to $V_i([0, 1-x])$ and $\textsc{Left-Mark}_i(0, k)$ on $\mathcal{I}$ is equal to $1 - \textsc{Right-Mark}_i(0, 1-k)$ on $\cev{\mathcal{I}}$ for any $k \in [0, 1]$, $\textsc{Left-Mark}_i(0, \textsc{Eval}_i(0, 1-x)-r)$ on $\mathcal{I}$ is equal to $1 - \textsc{Right-Mark}_i(0, 1-(V_i([0, 1-x])-r))$ on $\cev{\mathcal{I}}$.
    But $1-V_i([0, 1-x]) = 1 - \cev{V}_i([x, 1]) = \cev{V}_i([0, x])$, which means that the result is equal to $1 - \textsc{Right-Mark}_i(0, \cev{V}_i([0, x])+r)$.
    This is equal to $1 - \textsc{Right-Mark}_i(x, r) = 1-y$, which is the mirrored point of $y$.
\end{itemize}
This shows that if algorithm $\mathcal{A}$ determines the existence of a connected strongly-proportional allocation on $\cev{\mathcal{I}}$, then $\mathcal{B}$ determines the existence of a connected strongly-proportional allocation on $\mathcal{I}$.
Note that algorithm $\mathcal{B}$ makes at most two times the number of queries that $\mathcal{A}$ makes.
\end{proof}

\section{Proportionality}
\label{ap:proportionality}

With unequal entitlements, even a connected \emph{proportional} allocation may not exist.
\Cref{alg:general_improved} can be modified to determine the existence of a connected proportional allocation for agents with unequal entitlements (and to output one if it exists) by using \emph{left-marks} instead of right-marks.

\begin{algorithm}[ht]
  \caption{Determining the existence of a connected proportional allocation for $n$ agents.}
  \label{alg:general_proportional}
  \begin{algorithmic}[1]
    \State $x_{\emptyset} \leftarrow 0$
    \For{$k = 1, \ldots, n$}
      \For{each subset $N \subseteq [n]$ with $|N| = k$}
        \State $x_N \leftarrow \infty$
        \For{each agent $i \in N$}
          \State $y \leftarrow \text{\textcolor{red}{\textsc{Left}}-}\textsc{Mark}_i(x_{N \setminus \{i\}}, w_i)$
          \State \algorithmicif \ $y < x_N$ \algorithmicthen \ $x_N \leftarrow y$ \algorithmiccomment{this finds the ``best'' $x_N$}
        \EndFor
      \EndFor
    \EndFor
    \State \algorithmicif \ $x_{[n]}~\textcolor{red}{\leq}~1$ \algorithmicthen \ \Return true \algorithmicelse \ \Return false
  \end{algorithmic}
\end{algorithm}

\begin{proposition}
\Cref{alg:general_proportional} decides whether a connected proportional allocation exists for $n$ agents using at most $n \cdot 2^{n-1}$ queries.
\end{proposition}

We remark that our algorithm is similar to that in \citet[Theorem 4]{aumann2012computing} where they find an approximate optimum egalitarian welfare on a piece of cake---here, we extend it to agents with possibly unequal entitlements.

\end{document}